\documentclass[]{article}

\usepackage{amsmath}
\usepackage{amsfonts}
\usepackage{amssymb}
\usepackage{amsthm}
\usepackage{mathtools}
\usepackage{color}
\usepackage{esint}
\usepackage{hyperref}
\usepackage{enumitem}
\usepackage{appendix}
\usepackage{url}
\usepackage{titlesec}
\usepackage{hyperref}
\usepackage{chngcntr}
\usepackage{apptools}
\AtAppendix{\counterwithin{theorem}{section}}

\titleformat{\subsection}[runin]
{\normalfont\normalsize\bfseries}{\thesubsection}{.5em}{}

\titlespacing*{\subsubsection}
{0pt}{.5em}{.5em}

\titleformat{\subsubsection}[runin]
{\normalfont\normalsize\itshape}{\thesubsection}{.5em}{}

\theoremstyle{plain}
\newtheorem{theorem}{Theorem}

\newtheorem{proposition}[theorem]{Proposition}
\newtheorem{corollary}[theorem]{Corollary}
\newtheorem{lemma}[theorem]{Lemma}

\makeatletter
\def\thm@space@setup{\thm@preskip=2pt
	\thm@postskip=2pt}
\makeatother
\theoremstyle{plain}
\newtheorem{manualtheoreminner}{Theorem}
\newenvironment{manualtheorem}[1]{%
	\renewcommand\themanualtheoreminner{#1}%
	\manualtheoreminner
}{\endmanualtheoreminner}

\theoremstyle{definition}
\newtheorem{definition}[theorem]{Definition}
\newtheorem{example}[theorem]{Example}
\newtheorem{remark}[theorem]{Remark}

\newtheorem*{deff}{Definition}

\newcommand{\C}{\mathbb{C}}
\newcommand{\R}{\mathbb{R}}
\newcommand{\Z}{\mathbb{Z}}
\newcommand{\CP}{\mathbb{CP}}
\newcommand{\RP}{\mathbb{RP}}
\newcommand{\h}{\mathbb{H}}

\newcommand{\bz}{\bar{z}}
\newcommand{\bw}{\bar{w}}

\newcommand{\mF}{\mathcal{F}}
\newcommand{\bx}{\mathbf{x}}
\newcommand{\inn}{\langle \cdot, \cdot \rangle}
\newcommand{\p}{\partial}

\newcommand{\m}{\mathcal{M}_{0,4}}
\newcommand{\mn}{\mathcal{M}_{0, n}}
\newcommand{\tn}{\widetilde{\nabla}}
\newcommand{\MS}{\mathcal{MS}}

\newcommand{\MSm}{\mathcal{MS}_{(1,1)}^{(2)}(\alpha)}
\newcommand{\MSmm}{\mathcal{MS}_{(1,1)}^{(2)}(2\alpha)}

\newcommand{\ns}{\nabla^{\text{sph}}}
\newcommand{\nd}{\nabla^{\text{D}}}

\DeclareMathOperator{\Id}{Id}
\DeclareMathOperator{\tr}{tr}
\DeclareMathOperator{\hol}{(hol)}
\DeclareMathOperator{\Hol}{Hol}

\DeclareMathOperator{\re}{Re}
\DeclareMathOperator{\im}{Im}

\DeclareMathOperator{\Fix}{Fix}
\DeclareMathOperator{\dil}{dil}

\def\<#1>{\mathinner{\langle#1\rangle}}

\newcommand{\Address}{{
		\bigskip
		\begin{flushright}
			
			\textsc{King's College London, Department of Mathematics\\
				Strand, London, WC2R 2LS, United Kingdom} \\
			{\href{mailto:martin.deborbon@kcl.ac.uk}{\nolinkurl{martin.deborbon@kcl.ac.uk}}}, {\href{mailto:dmitri.panov@kcl.ac.uk}{\nolinkurl{dmitri.panov@kcl.ac.uk}}}
		\end{flushright}
}}

\title{Dunkl connections on \(\C^2\) and spherical metrics}
\date{\today}
\author{Martin de Borbon and Dmitri Panov}

\begin{document}

\maketitle

\begin{abstract}
	We show that general Dunkl connections on \(\C^2\) do not preserve non-zero Hermitian forms. Our proof relies on recent understanding of the non-trivial topology of the moduli space of spherical tori with one conical point.
\end{abstract}


\section{Introduction}\label{sec:intro}

K\"ahler metrics of constant holomorphic sectional curvature can be efficiently studied in terms of flat holomorphic connections that preserve appropriate non-zero Hermitian forms.
One instance of this occurs in the work of Couwenberg-Heckmann-Looijenga \cite{CHL} in which they produce new interesting families of Fubini-Study, flat, and complex hyperbolic K\"ahler metrics on complements of projective hyperplane arrangements. This provides a unified geometric framework that recovers as special cases the principal results of Deligne-Mostow on Lauricella hypergeometric functions and the work of Barthel-Hirzebruch-H\"ofer on line arrangements in the projective plane.

The above article develops around the central notion of a \emph{Dunkl connection} and it is shown that in the case of complex reflection arrangements the associated Dunkl connections preserve non-zero Hermitian forms. However, in the more general case where the arrangement is not of reflection type they write
\cite[p. 112]{CHL}:
 ``\emph{We do not know whether a Dunkl system with real exponents always admits a nontrivial flat Hermitian form, not even in the case dim V = 2} (\emph{the answer is probably no}).''

In this paper we settle this question and confirm this expectation. We consider Dunkl connections with simple poles at \(4\) complex lines and we show that most of them do not preserve any non-zero Hermitian form. Our proof is by contradiction.
If the statement were false then, using an analytic continuation argument, we can produce a section of the forgetful map from the moduli space of marked spherical tori with one conical point to the configuration space \(\m\) that would contradict the non-trivial topology of the space of spherical metrics established in \cite{EPM}.


\subsection{Main result.}\label{subsec:mainresult}
We work on \(\C^2\) with standard linear coordinates \(z, w\).
Let \(L_1, \ldots, L_n \subset \C^2\) be \(n \geq 3\) complex lines going through the origin.
We consider meromorphic connections \(\nabla\) on \(T\C^2\) that on the trivialization given by the coordinate vector fields \(\p_z, \p_w\)  have the form
\begin{equation}\label{eq:std1}
\nabla = d - \sum_i A_i \frac{d \ell_i}{\ell_i}	
\end{equation}
where \(A_i\) are non-zero \((2\times 2)\)-matrices of complex numbers and \(\ell_i\) are defining linear equations of \(L_i\). 
We also assume that 
\begin{equation}\label{eq:std2}
\forall i \hspace{2mm} \ker A_i = L_i \hspace{2mm}  \text{ and } \hspace{2mm} \sum_i A_i = c \cdot \Id \hspace{2mm} \text{ for some } c \in \C.	
\end{equation}
The geometric content of \eqref{eq:std2} is that \(\nabla\) is torsion free and flat, so it defines an affine structure on \(\C^2\) with certain singularities along the given lines.  

We say that \(\nabla\) is \emph{standard} if it is of the form given by Equations \eqref{eq:std1} and \eqref{eq:std2}. Let us now introduce the main objects of study of this paper.
\vspace{2mm}
\begin{deff}[{\cite[Definition 2.8]{CHL}}]
	A standard connection \(\nabla\) is Dunkl if there is a positive definite Hermitian inner product on \(\C^2\) that makes every residue matrix \(A_i\) self-adjoint with respect to it.
\end{deff}
\vspace{2mm}

Our first result is Proposition \ref{prop:Dunkl}. It shows that if we are given a collection of lines \(L_i\) and positive real numbers \(a_i>0\) then there exists a unique Dunkl connection with prescribed residue traces
\begin{equation}
	\tr A_i = a_i
\end{equation}
at \(L_i\) if and only if the `weights' \(a_i\) satisfy the `balancing condition'
\[
a_j < \sum_{i \neq j} a_i \hspace{2mm} \text{ for all } j.
\]

After establishing this existence and uniqueness result for Dunkl connections, we tackle the main question of whether they preserve non-zero Hermitian forms. To answer this question we exploit the natural correspondence between standard connections \(\nabla\) with \emph{unitary holonomy}  and \emph{spherical metrics} on \(\CP^1\) with cone angles \(2\pi(1-a_i)\) at the points corresponding to the lines \(L_i\).
This correspondence is defined by an elementary construction that locally replicates the making of the Fubini-Study metric on \(\CP^1\) out of the flat metric on \(\C^2\) by taking the K\"ahler quotient with respect to the  \(S^1\)-action given by scalar multiplication by complex units. We review this construction in Appendix \ref{sec:unitaryconnect}.

We restrict to the simple case of four lines and equal residue traces \(a_i=a_j\) for all \(i,j\).
Our main result is the next.

\begin{theorem}\label{thm:main}
	Suppose that \(n=4\) and \(a_i = a\) for all \(1 \leq i \leq 4\) and some \(a > 0\). Then, for \emph{generic} configurations of lines \(L_1, \ldots, L_4\) and values of \(a\), the Dunkl connection with \(\tr A_i =a\) at \(L_i\) doesn't preserve any non-zero Hermitian form.
\end{theorem}
\vspace{2mm}

\begin{remark}
	As a mater of fact all standard (and in particular all Dunkl) connections with  simple poles at \(3\) lines and real residue traces preserve non-zero Hermitian forms, see Section \ref{sec:3linesHF}. Therefore the case of \(4\) lines is in some sense the simplest. 
\end{remark}

\begin{remark}
	The more general case of  \(n \geq 4\)  lines can be deduced from the statement of Theorem \ref{thm:main} by taking arbitrary small residue traces at the other \(n-4\) lines, see Section \ref{sec:extension}. This fits in with the popular saying  `finding hay in a haystack', once we prove existence of one Dunkl connection that doesn't preserve any non-zero Hermitian form then it follows that most of them have this property. However, we don't write a single explicit example.
\end{remark}

\begin{proof}[Sketch proof of Theorem \ref{thm:main}]
	Assume that the statement is false. Then we run an \emph{analytic continuation argument} that produces a continuous family of spherical metrics \(g_{\lambda}\) on \(\CP^1\) with \(4\) cone points of the same angle \(2\pi\alpha\) at the configuration of points \(0, 1, \infty, \lambda\), where \(\alpha\) is a positive real number to be determined later. 
	
	We show that each metric \(g_{\lambda}\) is invariant under the action of the Klein \(4\)-group of M\"obius transformation that preserve the configuration of \(4\) cone points, so we can push it forward to a metric on \(\CP^1\) with \(4\) cone points of angles \(2\pi\alpha, \pi, \pi, \pi\). We take the elliptic curve \(C \xrightarrow{2:1} \CP^1\) that branches over these \(4\) points and pull back the metric to obtain a spherical torus \(\hat{g}_{\lambda}\) with one conical point of angle \(4\pi\alpha\). This way we obtain a section \(s: \lambda \mapsto \hat{g}_{\lambda}\) of the forgetful map 
	\begin{equation}\label{eq:Fmap}
	F: \MSmm \to  \m 
	\end{equation}
	where \(\MSmm\) is the moduli space of \emph{marked} spherical tori with one cone point of angle \(4\pi\alpha\). If \(\alpha>5/2\) and \(2\alpha \notin \Z\) then by \cite{EPM} the moduli space \(\MSmm\) is homeomorphic to a punctured compact orientable surface of genus \(\geq 1\). Using results in \cite{MPII} and \cite{EPM} we show that both \(F\) and the section \(s\)  extend continuously over the punctures. On the other hand, it is easy to see that a continuous map \(F: \Sigma \to S^2\) from a surface \(\Sigma\) of genus \(\geq 1\) to the \(2\)-sphere has no continuous right inverse \(F \circ s =1_{S^2}\). Therefore the existence of the section \(s\) gives us a contradiction and proves the theorem. 
\end{proof}

\subsection{Outline.}
The main work done in the paper is to carry out the analytic continuation argument in the beginning of the sketch proof above. This relies several results which are developed on the preliminary sections as detailed next.

In Section \ref{sec:Dunkl} we prove Proposition \ref{prop:Dunkl} which classifies Dunkl connections on \(\C^2\) in the case where all residue traces have the same sign. Our proof shows that the Dunkl connection depends \emph{analytically} on the configuration of lines and its residue traces. This analytic dependence is crucial in our argument.

In Section \ref{sec:holonomy} we give a description for the holonomy group of standard connections (see Lemma \ref{lem:mon}) and analyse invariant foliations. The main result is Proposition \ref{prop:flatbundle} which gives a numerical criterion in term of the residue traces \(a_i\) that guarantees irreducible holonomy. 

In Section \ref{sec:Dunkunitary} we study the Dunkl unitary case corresponding to reflection arrangements made of \(3\) and \(4\) lines. The main result is Proposition \ref{prop:dihedral} which identifies the values of \(a \in \R\) for which the Dunkl connection with equal residue traces \(a_i =a\) and simple poles at the dihedral \(B_2\)-arrangement has unitary holonomy. 

In Section \ref{sec:pf} we state Theorem \ref{thm:main} in its precise version and complete the proof of it. The results of Section \ref{sec:Dunkl} imply that the space of Dunkl connections that preserve a non-zero Hermitian form make a real analytic subset of the configuration space of lines. The main point is to show that there is at least one Dunkl connection that does not preserve any non-zero Hermitian form. Here we argue by contradiction and assume that all Dunkl connections preserve non-zero Hermitian forms. Then we combine the results of Sections \ref{sec:holonomy} and \ref{sec:Dunkunitary} together to produce a family of Dunkl \emph{unitary} connections that leads to the desired family of spherical metrics and complete the sketched proof.
The facts needed on moduli of spherical surfaces are collected in
Appendix \ref{sec:sphmet}.



\subsection*{Acknowledgments}
We want to thank Eduard Looijenga for answering our questions. 
This work was supported by the
EPSRC Project EP/S035788/1.

\section{Dunkl connections on \(\C^2\)}\label{sec:Dunkl}
The main result of this section is Proposition \ref{prop:Dunkl}, which gives necessary and sufficient conditions for the existence and uniqueness of a Dunkl connection with prescribed residue traces of the same sign. We reduce this problem to the more familiar one of finding a conformal automorphism of \(S^2\) that puts the centre of mass of a given weighted configuration of points at the origin. We solve this using the variational method, by minimizing a suitable convex function in hyperbolic \(3\)-space.\footnote{For the more general case of non-positively curved symmetric spaces see \cite{KLM}.}

\subsection{Dunkl inner products.}
Let \(\nabla\) be a Dunkl connection given by Equations \eqref{eq:std1} and \eqref{eq:std2}. Let \(\inn_H\) be a positive definite Hermitian inner product in \(\C^2\) such that the residue matrices \(A_i\) are self-adjoint with respect to it. This implies that the residue traces
\begin{equation}\label{eq:aidef}
a_i = \tr A_i 
\end{equation}
are real numbers and, if we denote by \(P_i\) the projection to the orthogonal complement of \(L_i\) determined by \(\inn_H\), then
\[
A_i = a_i \cdot P_i .
\]
This motivates the next.
\vspace{2mm}
\begin{definition}
	Let \(L_i \subset \C^2\) be \(n \geq 3\) complex lines going through the origin and let \(a_i\) be non-zero real numbers. A \emph{Dunkl inner product} adapted to \((L_i, a_i)\) is a positive definite Hermitian form \(\inn_H\) on \(\C^2\) such that
	\begin{equation}\label{eq:dunkl}
	\sum_i a_i \cdot P_i = c \cdot \Id 
	\end{equation}
	where \(P_i\) is the orthogonal projection to \(L^{\perp}\) given by \(\inn_H\). 
\end{definition}
\vspace{2mm}
\begin{remark}\label{rmk:pos}
	If \(\inn_H\) is a Dunkl inner product adapted to \((L_i, a_i)\) then it is also adapted to \((L_i, t  a_i)\) for every \(t \in \R^*\). 
\end{remark}



The main result of this section is the following.

\begin{proposition}\label{prop:Dunkl}
	Suppose that \(a_i \in \R^*\) have all the same sign. Then
	there is a Dunkl inner product \(\inn_H\) adapted to \((L_i, a_i)\) if and only if
	\begin{equation}\label{eq:stability}
		|a_j| < \sum_{i \neq j} |a_i| \hspace{3mm} \forall j .
	\end{equation}
	Furthermore, whenever \eqref{eq:stability} holds, the inner product \(\inn_H\) is unique up to scalar.
\end{proposition}
\vspace{2mm}

The proof of this result is carried out in Section \ref{se:pfdun}, the sketch is as follows. The space of rays of positive definite Hermitian inner products in \(\C^2\) is hyperbolic \(3\)-space \(\h^3\). The rank \(1\) orthogonal projections make the unit sphere at infinity of the Poicar\'e ball model of \(\h^3\). The lines \(L_i\) define points \(\bx(L_i)\) in \(S^2\) (see Lemma \ref{lem:sphere}) and we take the hyperbolic barycentre of the points \(\bx(L_i)\) with weights \(a_i\). This barycentre exists precisely when Equation \eqref{eq:stability} is satisfied (see Lemma \ref{lem:Fproperness}) and it defines the desired inner product \(\inn_H\). 

\subsection{Hermitian \((2 \times 2)\)-matrices.}
The space of Hermitian \((2\times2)\)-matrices
\[ H = \begin{pmatrix}
	r & t \\
	\bar{t} & s
	\end{pmatrix} 
\hspace{3mm}
r, s \in \R, \hspace{2mm} t \in \C 
\]
is a \(4\)-dimensional real vector space. The basis
\begin{equation*} \label{eq:basis}
	\sigma_0 =
\begin{pmatrix}
	1 & 0 \\
	0 & 1
\end{pmatrix},
\hspace{1mm}
\sigma_1 =
\begin{pmatrix}
	-1 &  0\\
	0 & 1
\end{pmatrix},
\hspace{1mm}
\sigma_2 =
\begin{pmatrix}
	0 & -1 \\
	-1 & 0
\end{pmatrix},
\hspace{1mm}
\sigma_3 =
\begin{pmatrix}
	0 & -i \\
	i & 0
\end{pmatrix} 
\end{equation*}
gives us linear coordinates \(H = \sum_i x_i \sigma_i\) with
\begin{equation}\label{eq:xcoord}
	r = x_0 - x_1, \hspace{2mm} s = x_0 + x_1, \hspace{2mm} -t = x_2 + i x_3 .
\end{equation}


\begin{lemma}\label{lem:sphere}
	Let \(L \subset \C^2\) be a complex line going through the origin and write \(P_{L^{\perp}}\) for the Hermitian matrix given by orthogonal projection to \(L^{\perp}\) with respect to the standard Euclidean inner product. Define
	\begin{equation}
		\bx(L) = 2 \cdot \Pi(P_{L^{\perp}}) 
	\end{equation}
	where \(\Pi(x_0, x_1, x_2, x_3) = (x_1, x_2, x_3)\) is the trace-free part projection.
	\begin{itemize}
		\item[(i)] If \(L=\C \cdot (\lambda, 1)\) then
		\begin{equation}\label{eq:sterproj}
			\bx(L) = \frac{1}{1+|\lambda|^2} (|\lambda|^2-1, 2 \re(\lambda), 2\im(\lambda)) .
		\end{equation}
		Thus the map \(L \mapsto \bx(L)\) defines the conformal bijection between \(\CP^1\) and the unit sphere \(S^2 \subset \R^3\) given by stereographic projection.
		\item[(ii)] The sum \(\sum a_i P_{L_i^{\perp}}\) is a constant multiple of the identity if and only if
		\begin{equation}\label{eq:barycenter}
			\sum_i a_i \cdot \bx(L_i) = 0 \in \R^3 .
		\end{equation}
	\end{itemize}
\end{lemma}

\begin{proof}
	\((i)\) Take a complex line \(L = \C \cdot (z, w) \subset \C^2\) with \(|z|^2+|w|^2=1\). Orthogonal projection to \(L^{\perp}\) is given by the Hermitian matrix
	\begin{equation}\label{eq:orthogonalproj}
		P_{L^{\perp}}
		=
		\begin{pmatrix}
			|w|^2 & - z \bw \\
			-\bz w & |z|^2
		\end{pmatrix} .	
	\end{equation}
	Using Equation \eqref{eq:xcoord} to compute the \(x_1,x_2,x_3\) coordinates of \(P_{L^{\perp}}\) we obtain that
	\begin{equation}\label{eq:hopf}
	\bx(L)=(|z|^2-|w|^2, 2\re(z \bar{w}), 2 \im(z\bar{w})) .
	\end{equation}
	The right hand side of \eqref{eq:hopf} is the Hopf map \(S^3 \to S^2\) and when applied to the unit vector \((z, w)=(1+|\lambda|^2)^{-1/2}(\lambda, 1)\)
	it gives \eqref{eq:sterproj}.
	
	\((ii)\) The sum \(S = \sum a_i P_{L_i^{\perp}}\) is a constant multiple of the identity if and only if \(\Pi(S)=0\). The statement follows because  \(\Pi\) is linear.
\end{proof}

\begin{corollary}\label{cor:stab}
	If there is a Dunkl inner product \(\inn_{H}\) adapted to \((L_i, a_i)\) then
	\begin{equation}
		|a_j| < \sum_{i \neq j} |a_i| \hspace{2mm} \text{ for all } j .
	\end{equation}
\end{corollary}

\begin{proof}
	By a linear change of coordinates we can assume that \(\inn_{H}\) is the standard Euclidean inner product. Equation \eqref{eq:barycenter} together with the triangle inequality imply \(|a_j| \leq \sum_{i \neq j} |a_i|\)
	with equality only when the points \(\bx(L_i)\) are aligned in \(\R^3\). Since \(3\) or more points in the unit sphere can never be aligned the strict inequality follows.
\end{proof}

\subsection{Hyperbolic \(3\)-space.}
The determinant defines a non-degenerate quadratic form of signature \((1,3)\) on the space of \((2\times 2)\)-Hermitian matrices. In linear coordinates given by Equation \eqref{eq:xcoord} 
\[rs - |t|^2 = x_0^2 - x_1^2 -x_2^2 - x_3^2 .\]
The set of unit determinant positive definite Hermitian matrices 
\begin{equation*}
	\h^3
	= \{
	H | \det H = 1, \hspace{1mm} \tr H > 0
	\}
\end{equation*}
equipped with the Riemannian metric induced by the quadratic form \(-\det\) make the hyperboloid model of hyperbolic \(3\)-space. The advantage of this description is that it makes explicit the isometric action of \(PSL(2, \C)\) on \(\h^3\)  which realizes the special isomorphism \(PSL(2, \C) \cong SO^+(1,3)\), as we explain next.

\subsubsection*{\(PSL(2, \C)\)-action.}
We follow \cite[Chaper 2.6]{ThuBook}. We identify Hermitian forms \(\inn_H\) with Hermitian matrices \(H\) by
\begin{equation}\label{eq:Hid}
	\langle v, w \rangle_H  = \langle H v, w \rangle .
\end{equation}
where \(\langle v,w \rangle = v_1 \bar{w}_1 + v_2 \bar{w}_2 \) is the usual Hermitian inner product of \(\C^2\).

Let \(A \in SL(2, \C)\) and let \(\inn_H\) be a positive definite Hermitian inner product, we define \(\inn_{A \cdot H}\)  by \(\langle v, w \rangle_{A \cdot H} = \langle A v, Aw \rangle_H\).
By Equation \eqref{eq:Hid} \(\langle Av, Aw \rangle_{H} = \langle HA v, Aw \rangle\) and the action on Hermitian matrices is 
\begin{equation}\label{eq:action}
H \mapsto A^* H A 
\end{equation}
where \(A^*\) is the conjugate transpose of \(A\).
Since \(\det(A^* H A) = |\det A|^2 \det H\)  the \(PSL(2, \C)\)-action preserves the quadratic form \(\det\) and so it preserves the hyperbolic metric.

\subsubsection*{Poincar\'e ball model.}
This is the unit ball \(\{\|x\| < 1\} \subset \R^3\) equipped with the conformal metric
\[
\frac{4}{(1-\|x\|^2)^2} |dx|^2 .
\]
It is equivalent to \(\h^3\) by stereographic projection from the point \((-1,0,0,0)\) to the hyperplane \(\{x_0=0\}\).

\subsubsection*{Busemann functions.}
Let \(x \in S^2\) be a point in the sphere at infinity of hyperbolic 3-space and let
\(\gamma_{x}\) be a unit speed geodesic ray in \(\h^3\) that converges to \(x\) as \(t \to +\infty\). The Busemann function \(b_x\) is defined as
\[b_x (y) = \lim_{t \to \infty} (d(y, \gamma_{x}(t))-t) . \]
Up to an additive constant, \(b_x\) is independent of the choice of ray. It is smooth, convex and \(1\)-Lipschitz. Its gradient is given by 
\begin{equation}\label{eq:derivative}
\nabla b_x (p)= - \gamma_{p, x}'(0)	
\end{equation}
where \(\gamma_{p, x}\) is the unique unit speed geodesic ray such that \(\gamma_{p, x}(0)=p\) and  \(\lim_{t \to +\infty} \gamma_{p, x}(t) = x\).

In the ball model and with the normalization \(b_x(0)=0\), we have an explicit expression \cite[p. 273]{BH}
\begin{equation}\label{eq:buseman}
b_x(y) = - \log \left(\frac{1-\|y\|^2}{\|x-y\|^2}\right) .	
\end{equation}

\begin{remark}
	Under the identification of \(\h^3\) with \(SL(2, \C)/SU(2)\) given by the \(SL(2, \C)\)-orbit of the identity matrix under the action \eqref{eq:action}, we have
	\[
	b_x (y) =  \log \|A v\|^2
	\]
	where \(A \in SL(2, \C)\) is such that \(y = A^*A \in \h^3\) and \(v\) is a unit vector that generates the complex line \(x \in \CP^1\). These convex functions are widely studied in Geometric Invariant Theory, see \cite{RT}.
\end{remark}

\subsection{Proof of Proposition \ref{prop:Dunkl}.} \label{se:pfdun}
Consider the convex function on \(\h^3\) given by
\begin{equation}
	F = \sum_i a_i \cdot b_{\bx(L_i)} .
\end{equation}
We assume that \(a_i>0\) for all \(i\), otherwise we just multiply all the \(a_i\)'s by \(-1\) (see Remark \ref{rmk:pos}).
\begin{lemma}\label{lem:convex}
	\(F\) is strictly convex.
\end{lemma}

\begin{proof}
	The Hessian of each Busemann function \(b_x\) is non-negative and has \(1\)-dimensional kernel given by the unit vector pointing towards \(x\). Since there are at least \(3\) points at infinity, the Hessian of \(F\) is strictly positive.
\end{proof}

\begin{lemma}\label{lem:Fproperness}
	If Equation \eqref{eq:stability} holds then \(F\) converges uniformly to \(+\infty\) at infinity. In other words, for any \(M>0\) there is a compact subset \(K \subset \h^3\) such that \(F(y) >M\) for all \(y \notin K\).
\end{lemma}

\begin{proof}
	Suppose not. Then there is \(M>0\) and a sequence \(y_n \in \h^3\) with \(r_n = d(y_n, 0) \to \infty\) such that \(F(y_n) \leq M\). Take unit speed geodesics with \(\gamma_n(0) =0\) and \(\gamma_n(r_n)=y_n\). Without loss of generality, we can assume that \(F(0)=0\). Since \(F\) is convex along \(\gamma_n\) we must have \(F(\gamma_n(t)) \leq M\) for all \(0 \leq t \leq r_n\).
	After taking a subsequence, the geodesics \(\gamma_n\) converge uniformly on compact subsets to a geodesic ray \(\gamma\) with 
	\begin{equation}\label{eq:bound}
	F(\gamma(t))\leq M	\hspace{2mm} \text{ for all } \hspace{2mm} t>0 .
	\end{equation}
	
	Let \(x \in S^2\) be the limit of \(\gamma(t)\) as \(t \to  \infty\) and consider the following \(2\) cases.  
	\begin{itemize}
		\item[(i)]  If \(x \neq \bx(L_i)\) for all \(i\) then each \(b_{\bx(L_i)} \to +\infty\) close to \(x\).
		\item[(ii)] If \(x= \bx(L_j)\) for some \(j\) then, using Equation \eqref{eq:derivative}, we see that \((F \circ \gamma)'(t)\) converges to \(-a_j + \sum_{i \neq j} a_i > 0\) as \(t \to  \infty\). By convexity \((F \circ \gamma)'(t)> \epsilon\) for some definite \(\epsilon>0\) and all \(t\) sufficiently large.
	\end{itemize}
	In either case (i) or (ii)  \(F(\gamma(t))\) is arbitrary large as \(t \to \infty\) which contradicts the upper bound \eqref{eq:bound}.
\end{proof}

\begin{proof}[Proof of Proposition \ref{prop:Dunkl}]
	If Equation \eqref{eq:stability} holds then by Lemma \ref{lem:Fproperness} the convex function \(F\) has a unique minimum at \(y \in \h^3\). Using the \(PSL(2, \C)\)-action we can set \(y=0\) in the Poincar\'e ball model. Then 
	\begin{equation}
		-\nabla F (0) = \frac{1}{2} \sum_i a_i \cdot \bx(L_i) 
	\end{equation}
	vanishes and existence follows from item \((ii)\) of Lemma \ref{lem:sphere}. Uniqueness is a direct consequence of strict convexity (Lemma \ref{lem:convex}).
\end{proof}

\section{Invariant foliations.}\label{sec:holonomy}
In this section we analyse the holonomy representation of flat connections of the form given by Equations \eqref{eq:std1} and \eqref{eq:std2}. We give numerical criteria that guarantee irreducible holonomy provided that certain non-integer conditions on the residue traces \(a_i\) are satisfied, see Corollaries \ref{cor:irrhol0} and \ref{cor:irrhol}. The Dunkl condition (that requires all the residue matrices to be self-adjoint with respect to some fixed inner product) is not used in the arguments. We work in the more general setting of \emph{standard connections}.

\subsection{Standard connections.}
Before recalling the general definition of standard connection we review the basic underlying example.

\begin{example}\label{ex:2cone}
	Let \(a \in \C\) and consider the connection on \(T\C\) that is given by  
	\begin{equation}\label{eq:2cone}
		\nabla = d - \frac{a}{z} dz 
	\end{equation}
	with respect to the trivialization given by the linear coordinate vector field \(\p_z\).\footnote{Since \(\nabla\) is holomorphic outside the origin and the complex dimension of the base is equal to \(1\), this automatically implies that \eqref{eq:2cone} is torsion free and flat.}
	The holonomy of this connection about a positive simple loop around the origin is equal to scalar multiplication by \(\exp(2\pi i a)\). 
	
	If \(a \in \R\) then \(\nabla\) is the Levi-Civita connection of a flat K\"aher metric \(g=|z|^{-2a}|dz|^2\) on \(\C\) whose behaviour at the origin depends on the values of \(a\) in the following manner. \((i)\) If \(a<1\) then \(g\) is isometric to a \(2\)-cone of total angle \(2\pi (1-a)\) with vertex at \(0\). \((ii)\) If \(a=1\) then \(g\) is isometric to a cylinder \(S^1 \times \R\) with its two ends at \(0\) and \(\infty\). \((iii)\) If \(a>1\) then \(g\) is isometric to a \(2\)-cone of total angle \(2\pi(a-1)\) with its infinite end at \(0\) and vertex at \(\infty\).
\end{example}

Standard connections are analogues of \eqref{eq:2cone} in higher dimensions.
We recall the set-up from the introduction.
Let \(L_1, \ldots, L_n \subset \C^2\) be \(n \geq 3\) distinct complex lines going through the origin with defining linear equations \(L_i=\{\ell_i=0\}\). Let \(\nabla\) be a connection on \(T\C^2\) that on the trivialization by linear coordinate vector fields \(\p_z, \p_w\)  has the form
\begin{equation}\label{eq:standard1}
\nabla = d - \sum_i A_i \frac{d \ell_i}{\ell_i}	
\end{equation}
where \(A_i \in M(2 \times 2, \C)\) are non-zero constant matrices.

\begin{remark}
	Connections of the form \eqref{eq:standard1} are characterized by having simple poles and being invariant under multiplication by scalars \(\lambda \in \C^*\). Indeed, if a connection \(\tn\) has these two properties then it must have constant residues \(A_i\) along a collection of complex lines \(L_i\). The difference \(\tn - \nabla\) where \(\nabla\) is as in Equation \eqref{eq:standard1} is a matrix of holomorphic \(1\)-forms invariant under scalar multiplication. Therefore this difference must vanish and \(\tn = \nabla\).  
\end{remark}

\begin{definition}
	We say that \(\nabla\) is \emph{standard} if 
	\begin{equation}\label{eq:standard2}
	\forall i \hspace{2mm} \ker A_i = L_i \hspace{2mm}  \text{ and } \hspace{2mm} \sum_i A_i = c \cdot \Id \hspace{2mm} \text{ for some } c \in \C.	
	\end{equation}
\end{definition}

\begin{remark}
	By \cite[Poposition 2.3]{CHL} (see also \cite[Proposition 4.10]{Pan}) 
	 \(\forall i\)  \(\ker A_i = L_i\) \(\iff\) \(\nabla\) is torsion free; and \(\sum_i A_i = c \cdot \Id\) \(\iff\) \(\nabla\) is flat.
	Therefore standard connections define \(\C^*\)-invariant affine structures on \(\C^2\) with `logarithmic singularities' along the lines \(L_i\).
\end{remark}

We write \(a_i = \tr A_i\) for the residue traces. 
As we saw in Section \ref{sec:Dunkl}, a Dunkl connection with prescribed residue traces exists and is unique when the conditions of Proposition \ref{prop:Dunkl} are met. As for standard connections we have the following. 

\begin{lemma}
	The set of all standard connections with prescribed residue traces \(a_i \in \C^*\) make an affine complex space of dimension \(n-3\). 
\end{lemma}

\begin{proof}
	If \(L \subset \C^2\) is a complex line then the vector subspace \(V_{L} \subset \C^4\) of all \((2\times 2)\)-matrices \(A\) such that \(L \subset \ker A\) has complex dimension \(2\) and the trace \(A \mapsto \tr A\) defines a non-zero linear function on \(V_L\). Let \(W\) be the subspace of \(\oplus_i V_{L_i}\) given by all tuples \((A_1, \ldots, A_n)\) such that \(\sum_i A_i\) is a constant multiple of the identity. We claim that \(W\) has complex dimension \(2n-3\).
	
	Without loss of generality we can assume that the first \(3\) lines are spanned by \((0,1), (1,0)\) and \((1,1)\) so we can write
	\[
	A_1 = \begin{pmatrix}
	x_1 & 0 \\
	y_1 & 0
	\end{pmatrix}, \hspace{2mm}
	A_2 = \begin{pmatrix}
	0 & y_2 \\
	0 & x_2
	\end{pmatrix}, \hspace{2mm}
	A_3 = \frac{1}{2}\begin{pmatrix}
	x_3 + y_3 & -x_3 - y_3 \\
	y_3-x_3 & x_3-y_3
	\end{pmatrix}
	\]
	for complex parameters \(x_i, y_i\) where \(x_i = \tr A_i\) for \(i=1, 2, 3\). The condition that \(\sum_{i=1}^n A_i\) is a constant multiple of the identity gives us \(3\) linearly independent equations that can be written as
	\[
	\begin{pmatrix}
	1 & 0 & 1/2 \\
	0 & 1 & -1/2 \\
	0 & 0 & 1
	\end{pmatrix}
	\cdot
	\begin{pmatrix}
	y_1 \\
	y_2 \\
	y_3
	\end{pmatrix}
	=
	\begin{pmatrix}
	v_1 \\
	v_2 \\
	v_3
	\end{pmatrix}
	\]
	where \(v_i\) are linear expressions on the variables \(x_1, x_2, x_3\) and the entries of \(A_i\) for \(i \geq 4\). Therefore the subspace \(W\) is parametrized by \(x_1, x_2, x_3\) together with the \(2(n-3)\) extra parameters accounting for the matrices \(A_i\) for \(i \geq 4\). This proves  the claim that \(\dim_{\C} W = 2n -3\).
	
	Finally, we note that the linear map \(W \to \C^n\) given by
	\[
	(A_1, \ldots, A_n) \mapsto (\tr A_1, \ldots, \tr A_n)
	\]
	is surjective and the statement of the lemma follows from this.
\end{proof}


\noindent\emph{Euler vector field.} A standard connection is invariant under scalar multiplication, this fact is reflected in the existence of an Euler vector field. This is a fundamental object in the study of standard connections and we discuss it next.

Let \(\nabla\) be a standard connection given by Equations \eqref{eq:standard1} and \eqref{eq:standard2} and let
\[
e = z \frac{\p}{\p z} + w \frac{\p}{\p w}
\]
be the Euler vector field of \(\C^2\). Taking traces in the equation \(\sum_i A_i = c \cdot \Id\) we see that
\begin{equation}\label{eq:c}
	c = \frac{1}{2} \sum_i a_i .
\end{equation}
If \(c \neq 1\) then we define the Euler vector field \(E\) of the connection \(\nabla\) as
\begin{equation}
	E = \frac{1}{1-c} \cdot e .
\end{equation}
The reason we multiply \(e\) by the factor \(1/(1-c)\) is given by the following.

\begin{lemma}
	If \(c \neq 1\) then
	\begin{equation}\label{eq:nablaE}
	\nabla E = \Id ,
	\end{equation}
	meaning that \(\nabla_v E = v\) for all tangent vectors \(v\). Otherwise, if \(c=1\) then \(\nabla e = 0\).
\end{lemma}

\begin{proof}
	We follow the proof of \cite[Proposition 2.2]{CHL}.
	Since \(\ker A_i = L_i\), we can write \(A_i = d\ell_i \otimes n_i \) where \(n_i\) spans the image of \(A_i\) and therefore \(A_i(e)= d\ell_i (e)  n_i = \ell_i  n_i\). Then
	\begin{equation}\label{eq:nablae}
	\nabla e = d(e)- \sum_i A_i(e) \frac{d\ell_i}{\ell_i} = \Id - \sum_i A_i
	= (1-c) \cdot \Id .	
	\end{equation}
	The statement of the Lemma follows from Equation \eqref{eq:nablae}.
\end{proof}

\begin{remark}\label{rmk:flatherm}
	We take the chance to reproduce here an observation made in \cite[Section 3.2]{CHL} that concerns the case of standard connections for which \(c=1\). It goes as follows.
	
	If \(\nabla\) is a standard connection then it induces a connection on the line bundle of holomorphic volume forms \(\Lambda^2 T^*\C^2\) that on the trivialization given by \(dz \wedge dw\) is given by
	\[
	d + \left( \sum_i a_i \frac{d\ell_i}{\ell_i} \right) .
	\]
	The locally defined section \(s=\left(\prod_i \ell_i^{-a_i} \right)dz \wedge dw\) is parallel and if all \(a_i\) are real then \(s \wedge \bar{s}\) is a globally defined parallel real volume form.
	
	Suppose now that all \(a_i\) are real and that \(\sum_i a_i = 2\). Then the vector field \(e=z\p_z+w\p_w\) is parallel and the contraction of \(s\) with \(e\) is a locally defined (or multivalued) parallel \(1\)-form \(\eta\).
	The product \(h=\eta \otimes \bar{\eta}\) is `uni-valued' and it defines a non-zero Hermitian form that is preserved by \(\nabla\). The Hermitian form \(h\) is degenerate and has kernel spanned by \(e\), indeed it is the pull-back of the flat metric on \(\CP^1\) given by
	\[
	\left(\prod |\xi - p_i|^{-a_i} \right) |d\xi| .
	\]
	This flat metric on \(\CP^1\) has cone angles \(2\pi\alpha_i\) at \(p_i\) where \(\alpha_i=1-a_i\).\footnote{If \(\alpha_i=0\) then the metric has an infinite cylindrical end at \(p_i\) while if \(\alpha_i<0\) then the metric has an infinite cone end of cone angle \(-2\pi\alpha_i\) at \(p_i\), see Example \ref{ex:2cone}.} In particular, we see that in our main Theorem \ref{thm:main} we need to take not only the configuration of lines but also the residue trace \(a\) to be generic as well.
\end{remark}

\subsection{Holonomy representation.}
In this section we give a rough description for the holonomy group of standard connections (Lemma \ref{lem:mon}) and as a consequence we deduce a criterion for irreducible holonomy (Corollary \ref{cor:irrhol0}). 

We begin by describing the fundamental group of the lines complement. The space \(\C^2 \setminus \cup_i L_i\) retracts to the complement of \(n\) Hopf circles in \(S^3\). Since the Hopf fibration restricts to a trivial circle bundle over a punctured sphere and the fundamental group of a sphere with \(n\) punctures is the free group \(F_{n-1}\) on \(n-1\) generators; we can easily see that \(\pi_1(\C^2 \setminus \cup_i L_i)\) is isomorphic to the direct product \(\Z \times F_{n-1}\). We give a convenient set of generators and relations.

Fix \(x_0 \in \C^2 \setminus \cup_i L_i\). Let \(C\) be an affine complex line that goes through \(x_0\) and intersects the lines \(L_i\) at \(n\) different points \(c_i\).  We take \emph{canonical loops} (see \cite[Definition 18.1]{Yak}) \(\gamma_1, \ldots, \gamma_n\) generating the fundamental group of the punctured line \(C \setminus \cup_i \{c_i\}\) where each
\(\gamma_i\) is a loop based at \(x_0\) and contained in \(C\) that winds counter-clockwise around \(c_i\) and the product \(\gamma_1 \ldots \gamma_n\) is homotopic to a simple positive loop that surrounds all the points \(c_i\).
\begin{lemma}\label{lem:fundgroup}
	The fundamental group \(\pi_1(\C^2 \setminus \cup_i L_i)\) is generated the loops \(\gamma_i\) subject to the relations
	\begin{equation}\label{eq:relations}
		[\gamma_i, \gamma_1 \ldots \gamma_n] = 1 \hspace{2mm} \text{ for all } 1 \leq i \leq n .
	\end{equation}
\end{lemma}

\begin{proof}[Sketch proof.]
	This is standard and proved in {\cite[p. 161]{yoshida}} for example.
	We limit ourselves to explain Equation \eqref{eq:relations}.
	
	Let \(C_0\) be the complex line that goes through the origin and \(x_0\)
	and let \(\gamma_0(t) = \exp(2\pi i t) \cdot x_0\) be the simple positive loop around the origin contained in \(C_0\). We show that \(\gamma_0\) is in the centre of \(\pi_1(\C^2 \setminus \cup_i L_i)\). If \(\beta\) is a loop based at \(x_0\) then we have a map from the torus \(S^1 \times S^1\) to \(\C^2 \setminus \cup_i L_i\) given by \((s, t) \mapsto \exp(2\pi i t) \beta(s)\) that restricts to \(\gamma_0\) on \(\{1\} \times S^1\) and to \(\beta\) on \(S^1 \times \{1\}\). 
	Since the fundamental group of the torus is abelian we conclude that \([\beta, \gamma_0]=1\).
	On the other hand, the loop \(\gamma_0\) is homotopic to the product \(\gamma_1 \ldots \gamma_n\) and therefore the relations \eqref{eq:relations} hold.	
\end{proof}

Fix a point \(x_0\) outside the lines and write \(V=T_{x_0}\C^2\). We consider the holonomy group obtained by parallel transport along loops based at \(x_0\)
\[
\Hol(\nabla) \subset GL(V) .
\]

\begin{lemma}\label{lem:mon}
	Suppose that the residue traces \(a_i=\tr A_i \in \C \setminus \Z\) for all \(i\). Then the holonomy of \(\nabla\) is generated by \(M_1, \ldots, M_n \in GL(V)\) where each \(M_i\) is diagonalizable with eigenvalues \(1\) and \(\exp(2\pi i a_i)\) and their product satisfies
	\begin{equation}\label{eq:monrel}
	M_1 \ldots M_n = \exp(2\pi i c) \cdot \Id  
	\end{equation}
	where \(c\) is given by Equation \eqref{eq:c}.
\end{lemma}

\begin{proof}
	Take generators of the fundamental group \(\gamma_1, \ldots, \gamma_n\) that lie on a complex line \(C\) that intersects \(\cup_iL_i\) at \(n\) distinct points \(c_i\) as in Lemma \ref{lem:fundgroup} and let \(M_i \in GL(V)\) be the holonomy of \(\nabla\) at \(\gamma_i\).
		
	To compute \(M_i\) we restrict \(\nabla\) to the complex line \(C\). More precisely, we parametrize \(C\) by \(F: t \mapsto x_0 + t \cdot v\) where \(v\) is a tangent direction of \(C\) and we pull-back \(\nabla\) to the complex plane \(\C\) using this parametrization \(F\). Note that \(F^*(d \ell_i / \ell_i) = dt / (t - \xi_i)\) with \(F(\xi_i)=c_i\) and the pull-back \(F^*\nabla\) is the meromorphic connection on \(\C\) given by
	\begin{equation}\label{eq:fuchsian}
		d - \sum_i A_i \frac{dt}{t-\xi_i} .
	\end{equation}
	Close to \(\xi_i\) we can write \eqref{eq:fuchsian} as \(d- \left( A_i dt/(t-\xi_i) + \hol \right)\) where \(\hol\) is a holomorphic term.
	The eigenvalues of the residue matrix \(A_i\) at \(\xi_i\) are \(0\) and \(a_i\). Since \(a_i \notin \Z\), the normal form theorem for non-resonant Fuchsian singularities \cite[Theorem 16.16]{Yak} implies that \(M_i\) is conjugate to
	\[
	\exp \left( 2\pi i
	\begin{pmatrix}
	a_i & 0 \\
	0 & 0
	\end{pmatrix}
	\right)
	=
	\begin{pmatrix}
	\exp(2\pi a_i) & 0 \\
	0 & 1
	\end{pmatrix} .
	\]
	
	To compute the holonomy at \(\gamma_1 \ldots \gamma_n \sim \gamma_0\) we restrict \(\nabla\) to the complex line through the origin \(C_0\) (notation as in the proof of Lemma \ref{lem:fundgroup}) parametrized by \(t \mapsto t x_0\). The pull-back connection on is equal to the Euler system on \(\C\)
	\[
	d - \begin{pmatrix}
	c & 0 \\
	0 & c
	\end{pmatrix}
	\frac{dt}{t} ,
	\]
	where we have used that \(\sum_i A_i = c \cdot \Id\). The holonomy of this Euler system at \(\gamma_0\) is equal to \(\exp(2\pi i c) \cdot \Id\) (as in Example \ref{ex:2cone}) and Equation \eqref{eq:monrel} follows. 
\end{proof}

\begin{corollary}\label{cor:irrhol0}
	Suppose that \(\nabla\) is a standard connection such that the following non-integer conditions on the residue traces hold:
	\begin{enumerate}[label=\upshape(\Roman*)]
		\item\label{noin1} \(a_i \notin \Z\) for all \(i\);
		\item\label{noin2}  \(\sum_{i \in I} a_i - \sum_{i \notin I} a_i \notin 2\Z\) for every \(I \subset \{1, \ldots, n\}\).
	\end{enumerate}
	Then \(\nabla\) has irreducible holonomy.
\end{corollary}

\begin{proof}
	Since \(a_i \notin \Z\) we can apply Lemma \ref{lem:mon} to obtain holonomy endomorphisms \(M_i \in \Hol(\nabla)\) and direct sum decompositions \(T_{x_0}\C^2 = E_{i1} \oplus E_{i2}\) where \(M_i\) acts by \(1\) on \(E_{i1}\) and by
	\(\exp(2\pi i a_i) \neq 1\) on \(E_{i2}\). 
	
	Suppose by contradiction that \(L \subset T_{x_0}\C^2\) is a non-zero subspace invariant by \(\Hol(\nabla)\). Then we must have \(L= E_{i1}\) or \(L = E_{i2}\) for every \(i\). Let \(I\) be the set of indices \(i\) such that \(L = E_{i2}\). If \(v\) is a non-zero vector in \(L\) then
	\begin{equation}\label{eq:irr1}
		M_1 \ldots M_n v = \exp \left(2\pi i \sum_{i \in I} a_i \right) v .		
	\end{equation} 
	On the other hand, Equation \eqref{eq:monrel} gives us
	\begin{equation}\label{eq:irr2}
	M_1 \ldots M_n v = \exp(2\pi i c) \cdot v  .
	\end{equation}
	Equations \eqref{eq:irr1} and \eqref{eq:irr2} imply that \((\sum_{i \in I} a_i) - c \in \Z\). Using that \(c=(1/2)\sum_i a_i\) we see that \(\sum_{i \in I} a_i - \sum_{i \notin I} a_i \in 2\Z\) and Condition \ref{noin2} is violated.
\end{proof}

\subsection{Flat distributions.}
The main result of this section is the next Proposition \ref{prop:flatbundle} and its Corollary \ref{cor:irrhol}. They refine Corollary \ref{cor:irrhol0}, weakening its Condition \ref{noin2} by allowing the possibility that  \(\sum_{i \in I} a_i - \sum_{i \notin I} a_i = 0\). This is important for our application in Section \ref{sec:pf} where we consider the case when all \(a_i\)'s are equal. Our proof uses differential geometry formulas and goes along the lines of \cite[Lemma 3.2]{CHL} (which also considers higher dimensions in the Dunkl setting).

\begin{proposition}\label{prop:flatbundle}
	Let \(\nabla\) be a standard connection of the form given by Equations \eqref{eq:standard1} and \eqref{eq:standard2}.
	Suppose that the residue traces \(\tr A_i=a_i \notin \Z\).
	If \(\nabla\) preserves a rank \(1\) distribution \(\mF \subset T(\C^2 \setminus \cup_i L_i)\)
	then at least one of the following must hold:
	\begin{itemize}
		\item[(i)] there is a subset of indices \(I\) such that \(\sum_{i \in I}a_i - \sum_{i \notin I}a_i\) is a \emph{positive} even integer;
		\item[(ii)] there is some \(j\) such that \(\pm a_j=\sum_{i \neq j}a_i\).
	\end{itemize}
\end{proposition}


Let's begin the proof of Proposition \ref{prop:flatbundle}.
Assume that \(\nabla\) preserves a rank \(1\) distribution  \(\mF \subset T(\C^2 \setminus \cup_i L_i)\) and that \(a_i \notin \Z\) for all \(i\). Let \(n_i\) be an eigenvector for \(A_i\) with eigenvalue \(a_i\)
\begin{equation}
A_i n_i = a_i \cdot n_i .	
\end{equation}
The residue matrix \(A_i\) has rank \(1\) and its image is equal to \(\C \cdot n_i\).

\begin{lemma}\label{lem:extension}
	\(\mF\) extends as a rank \(1\) distribution of \(T(\C^2 \setminus \{0\})\) with  \(\mF(x) = L_i\) or \(\mF(x) = \C \cdot n_i\) at \(x \in L_i \setminus \{0\}\).
\end{lemma}

\begin{proof}
	Take \(x \in L_i \setminus \{0\}\), since \(a_i \notin \Z\), by the normal form theorem for flat logarithmic connections\footnote{This result can be thought of as a family version of the normal form theorem for Fuchsian singularities, applied to transverse discs to the singular set.} \cite[Theorem A.11]{dBP} close to \(x\) there is a frame of holomorphic vector fields \(v_1\) (tangent to \(L_i\)) and \(v_2\) (transverse to \(L_i\)) such that 
	the connection \(\nabla\) in the frame \(v_1, v_2\) is
	\begin{equation}\label{eq:normform}
		\nabla = d - \begin{pmatrix}
		0 & 0 \\
		0 & a_i
		\end{pmatrix}
		\frac{d\ell_i}{\ell_i} .
	\end{equation}
	The vector fields \(v_1, v_2\) give a pair of foliations which decompose \(T \C^2 = \mF_1 \oplus \mF_2\)
	with \(\mF_1 = L_i\) and \(\mF_2 = \C \cdot n_i\) along \(L_i\).
	The holonomy of \(\nabla\) at a loop around \(L_i\) acts by \(1\) on \(\mF_1\) and by \(\exp(2\pi i a_i) \neq 1\) on \(\mF_2\) and  we  must have that either \(\mF = \mF_1\) or \(\mF_2\).	
\end{proof}

\begin{lemma}
	There is a homogeneous holomorphic vector field \(X\) on \(\C^2\) which is
	non-vanishing outside \(0\) and is such that \(\mF(x) = \C \cdot X(x)\) at every \(x \in \C^2 \setminus \{0\}\).
\end{lemma}

\begin{proof}
	Since \(\mF\) is invariant under the multiplicative action of \(\C^*\) on \(\C^2\) it defines a holomorphic map \(\Phi : \CP^1 \to \CP^1\). Write \(\Phi ([z,w]) = [P(z,w), Q(z,w)] \) where \(P\) and \(Q\) are homogeneous polynomials of the same degree with no common zero outside the origin.  Take \(X=P \p_z + Q \p_w	\).
\end{proof}

Let \(e = z \p_z + w \p_w \) be the Euler vector field of \(\C^2\). Then 
\begin{equation}\label{eq:eX}
[e, X] = dX	
\end{equation}
where \(d\) is the degree of the vector field \(X\), it is an integer \(\geq -1\) and it is equal to the degree of the polynomial components of \(X\) minus \(1\). Our goal is to compute \(d\) in terms of the \(a_i\)'s. In order to do this we use the torsion free formula
\begin{equation}\label{eq:torfree}
[e, X] = \nabla_e X - \nabla_X e .
\end{equation}
It follows from Equation \eqref{eq:nablae} that
\begin{equation}\label{eq:nablaeX}
	\nabla_X e = (1-c) \cdot X .
\end{equation}
We proceed to compute \(\nabla_e X\).

Since  \(\nabla \mF \subset \mF\), it follows that \(\nabla_v X\) is a multiple of \(X\) for every tangent vector \(v\). This way we obtain a holomorphic \(1\)-form \(\eta\) on \(\C^2 \setminus \cup_i L_i\) such that 
\begin{equation}\label{eq:eta}
\nabla_v X = \eta(v) \cdot X 
\end{equation}
for every \(v \in T_x \C^2\) and \(x \in \C^2 \setminus \cup_i L_i\).

\begin{lemma}\label{lem:etahom}
	The contraction of \(\eta\) with the Euler vector field \(e\) is a constant.  
\end{lemma}

\begin{proof}
	Using Equations \eqref{eq:eX}, \eqref{eq:torfree}, \eqref{eq:nablaeX} and \eqref{eq:eta} we obtain
	\begin{equation}\label{eq:d1}
		d = \eta(e) - (1-c) 
	\end{equation}
	and the statement follows.
\end{proof}

\begin{lemma}
	Let \(I\) be the subset of \(\{1, \ldots, n\}\) made of the indices \(i\)'s such that \(L_i\) is transverse to \(\mF\). Then
	\begin{equation}\label{eq:eta2}
	\eta =  - \sum_{i \in I} a_i \frac{d \ell_i}{\ell_i}  .
	\end{equation}
\end{lemma}

\begin{proof}
	Take \(x \in L_i \setminus \{0\}\). Equation \eqref{eq:normform} implies that close to \(x\), up to the addition of a holomorphic \(1\)-form, \(\eta\) is equal to \(-a_i d\ell_i / \ell_i\) if \(\mF\) is transverse to \(L_i\) and \(\eta\) is holomorphic if \(\mF\) is tangent to \(L_i\). By Hartogs, we can write
	\(\eta = \eta_{\hol} - \sum_{i \in I} a_i d \ell_i /  \ell_i \)
	where \(\eta_{\hol}\) is holomorphic on the whole of \(\C^2\). Since \(\nabla\) is flat \(\eta\) is closed, hence so is \(\eta_{\hol}\) and therefore we can write \(\eta_{\hol} = du\) with \(u\) a holomorphic function on \(\C^2\). 
	
	On the other hand, by Lemma \ref{lem:etahom} the derivative of \(u\) along the Euler field  \(du(e)\) is equal to a constant \(C\). Since \(e\) vanishes at \(0\), evaluating \(du(e)\) at the origin shows that \(C=0\).  Hence \(u\) is constant and \(du =\eta_{\hol}=0\).
\end{proof}

\begin{lemma}\label{lem:constantX}
	If the non-zero vector field \(X\) is constant then there is some \(j\) such that \(\pm a_j = \sum_{i \neq j} a_i\).
\end{lemma}

\begin{proof}
	By Lemma \ref{lem:extension} we must have \(\C \cdot X = L_i\) or \(\C \cdot X = \C \cdot n_i\) for every \(i\).\footnote{In the Dunkl case \(\C \cdot n_i=L_i^{\perp}\) and such constant invariant distributions do not exists as the number of lines is \(\geq 3\).} We distinguish two cases.
	
	Case \(1\): \(\C \cdot X = \C \cdot n_i\) for every \(i\). Then the image of \(\sum_i A_i\) is contained in \(\C \cdot X\). Since \(\sum_i A_i = c \cdot \Id\) we must have \(c=(1/2)\sum_i a_i=0\) in this case. Therefore \(\sum_i a_i = 0\) and \(-a_j = \sum_{i \neq j} a_i\) for every \(j\).
	
	Case \(2\): \(X \in L_j\) for some \(j\). Since \(X\) is non-zero we can't have \(X \in L_i\) for \(i\neq j\), therefore \(X\) belongs to the \(a_i\)-eigenspace \(\C \cdot n_i\) of \(A_i\) for every \(i \neq j\). This implies that
	\[
	c \cdot X = \left(\sum_i A_i\right) X = \left( \sum_{i \neq j} a_i \right) X .
	\]
	It follows that \(c= \sum_{i \neq j}a_i\). Using that \(c=(1/2)\sum_i a_i\) we conclude that \(a_j = \sum_{i \neq j} a_i\).
\end{proof}

\begin{proof}[Proof of Proposition \ref{prop:flatbundle}]
	Equation \eqref{eq:eta2} implies that \(\eta(e)= - \sum_{i \in I} a_i\). Plugging this into Equation \eqref{eq:d1} and using \(c=(1/2)\sum_i a_i\), we conclude that the degree \(d\) of \(X\) is equal to
	\begin{equation}\label{eq:d2}
	d = - \sum_{i \in I} a_i + c -1 = \frac{1}{2} \left(\sum_{i \notin I} a_i - \sum_{i \in I} a_i \right) -1.	
	\end{equation}
	If \(d=-1\) then \(X\) is a non-zero constant vector field and Lemma \ref{lem:constantX} shows that item \((ii)\) holds. If \(d \geq 0\) then Equation \eqref{eq:d2} implies that \(\sum_{i \notin I} a_i - \sum_{i \in I} a_i\) is a positive even integer number and therefore item \((i)\) is satisfied.
\end{proof}

\begin{corollary}\label{cor:irrhol}
	Let \(\nabla\) be a standard connection with non-integer residue traces \(a_i \in \C \setminus \Z\) such that \(\sum_i a_i \neq 0\) and \(a_j \neq \sum_{i \neq
	 j} a_i\) for every \(j\). 
 \begin{itemize}
 	\item[(i)] If \(\nabla\) preserves a rank \(1\) distribution \(\mF\) then \(\sum_{i \in I} a_i - \sum_{i \notin I} a_i\) is a positive even integer, where \(I\) is the subset of \(i\)'s such that \(\mF\) is tangent to \(L_i\).
 	\item[(ii)] If there is no subset of indices \(I\) such that  \(\sum_{i \in I} a_i - \sum_{i \notin I} a_i\) is a positive even integer then the holonomy of \(\nabla\) is irreducible.
 \end{itemize}

\end{corollary}

\section{Flat Hermitian forms for the \(B_2\)-arrangement.}\label{sec:Dunkunitary}
The main result of this section is Proposition \ref{prop:dihedral}.
It identifies the values of \(a \in \R\) for which the Dunkl connection \(\nabla_a\) with equal residue traces \(\tr A_i = a\) and simple poles at the dihedral \(B_2\)-arrangement is unitary.

The \(4\) lines of the \(B_2\)-arrangement correspond to the symmetric configuration of points \(0, \infty, 1, -1 \in \CP^1\). We use the branched covering
\((z, w) \mapsto (z^2, w^2)\) (which induces the map \(\xi \mapsto \xi^2\) on the Riemann sphere) to write \(\nabla_a\) as a pull-back of a connection with simple poles at \(3\) lines. The existence of flat Hermitian forms for standard connections with simple poles at \(3\) lines is well-known and in Sections \ref{sec:2matrices} and \ref{sec:3linesHF} we present a self-contained account of it.

\subsection{Invariant Hermitian forms for a pair of matrices.}\label{sec:2matrices}
In this section we review the existence of Hermitian forms invariant under certain subgroups of \(GL(2, \C)\) generated by \(2\) elements. The main result is Lemma \ref{lem:hermform}. This is
classical in the context of Gauss' hypergeometric equation, see \cite{Beu}.
 
\begin{lemma}\label{lem:circle}
	Let \(R, S \in GL(2, \C)\) be diagonalizable matrices with unit eigenvalues \(r_1, r_2 \in S^1\) and \(s_1, s_2 \in S^1\). We assume that these eigenvalues are all different, that is \(r_1 \neq r_2, s_1 \neq s_2\) and \(\{r_1, r_2\} \cap \{s_1, s_2\} = \emptyset\). Moreover, suppose that \(1\) is an eigenvalue of \(RS^{-1}\). 
	
	Then the corresponding eigenspaces
	\[
	E_{r_1}, E_{r_2}, E_{s_1}, E_{s_2} \in \CP^1
	\]
	are all distinct and lie on a circle in \(\CP^1\).
\end{lemma}

\begin{remark}[Geometric interpretation]
	The matrices \(R, S\) act on hyperbolic \(3\)-space \(\h^3\) as elliptic isometries with axes of rotation \(\ell_R\) and \(\ell_S\). The condition that \(RS^{-1}\) has eigenvalue \(1\) implies that \(RS^{-1}\) is either elliptic (in case is diagonalizable) or parabolic (if it is not). Lemma \ref{lem:circle} asserts that the axes of rotation \(\ell_R\) and \(\ell_S\) lie on a hyperbolic plane \(\h^2 \subset \h^3\), see also \cite[Lemma 3.4.1]{KapMond}.
\end{remark}

\begin{proof}[Proof of Lemma \ref{lem:circle}]
	By hypothesis there is a non-zero vector \(\tilde{v} \in \C^2\) such that
	\(RS^{-1} \tilde{v} = \tilde{v}\). If we take \(v = S^{-1} \tilde{v}\) then
	\[
	Rv = RS^{-1} \tilde{v} = \tilde{v} = S v .
	\]  
	Note that \(w:= Rv = Sv\) is not a multiple of \(v\), otherwise the endomorphisms \(R, S\) would share an eigenvalue. We conclude that \(\{v, w\}\) form a basis of \(\C^2\) and so the matrices \(R, S\) are simultaneously conjugate to
	\begin{equation}\label{eq:monmat}
		\begin{pmatrix}
		0 & -r_1 r_2 \\
		1 & r_1 + r_2
		\end{pmatrix},
		\hspace{2mm}
		\begin{pmatrix}
		0 & -s_1 s_2 \\
		1 & s_1 + s_2
		\end{pmatrix}
	\end{equation}
	which have eigenvectors
	\begin{equation}\label{eq:moneigen}
		\begin{pmatrix}
		-r_2 \\
		1
		\end{pmatrix},
		\hspace{2mm}
		\begin{pmatrix}
		-r_1 \\
		1
		\end{pmatrix},
		\hspace{2mm}
		\begin{pmatrix}
		-s_2 \\
		1
		\end{pmatrix},
		\hspace{2mm}
		\begin{pmatrix}
		-s_1 \\
		1
		\end{pmatrix} . \qedhere
	\end{equation}
\end{proof}

Let \(\R^{1,3}\) be the \(4\)-dimensional real vector space of Hermitian matrices \(H\) with linear coordinates \((x_0, x_1, x_2, x_3)\) given by
\[
H =
\begin{pmatrix}
x_0 - x_1 & -x_2 - i x_3 \\
-x_2 + i x_3 & x_0 + x_1
\end{pmatrix} .
\]
Recall that Hermitian matrices \(H\) correspond to Hermitian forms via \(\langle v, w \rangle_H = \langle H v , w \rangle\)
where \(\inn\) is the usual inner product of \(\C^2\). Given \(R \in GL(2, \C)\) it acts on Hermitian forms by \(R \cdot \inn_{H} = \langle R \cdot, R \cdot \rangle_H\) and the corresponding action on Hermitian matrices is \(H \mapsto R^* H R\).

\begin{lemma}\label{lem:fixedH}
	Let \(R \in GL(2, \C)\) have unit eigenvalues \(r_1, r_2 \in S^1\) with \(r_1 \neq r_2\) and corresponding eigenspaces \(E_1, E_2\). Then the space of Hermitian forms invariant by \(R\) make a \(2\)-dimensional real subspace of \(\R^{1,3}\) spanned by the orthogonal projections to \(E_1^{\perp}\) and \(E_2^{\perp}\).
\end{lemma}

\begin{proof}
	The conjugate transpose \(R^*\) acts on \(E_1^{\perp}\) as complex multiplication by \(\bar{r}_2\) and on \(E_2^{\perp}\) as \(\bar{r}_1\). Indeed,
	let \(u_1, u_2\) unit norm eigenvectors with \(Ru_i = r_i u_i\) and take unit normals \(n_i \perp E_i\) then
	\[
	\langle R^* n_i , u_i \rangle = \langle n_i, Ru_i \rangle = 0
	\] 
	so \(R^*n_i \perp E_i\) and we must have \(R^*n_i = \lambda_i n_i\) for some \(\lambda_i \in \C\). To compute \(\lambda_1\) take the inner product
	\begin{align*}
		\langle R^* n_1, u_2 \rangle &= \lambda_1 \langle n_1, u_2 \rangle \\
		&= \bar{r}_2 \langle n_1, u_2 \rangle .
	\end{align*}
	Since \(\langle n_1, u_2 \rangle \neq 0\) (because \(u_2 \notin \C \cdot u_1\)) we conclude that \(\lambda_1=\bar{r}_2\). In the same way we obtain \(\lambda_2=\bar{r}_1\).
	
	Orthogonal projection to \(E_1^{\perp}\) gives us Hermitian form \(P_{E_1^{\perp}} (v, v) = | \langle v, n_1 \rangle |^2\) and
	\begin{align*}
	R \cdot P_{E_1^{\perp}} (v, v)  &=
	| \langle R v, n_1 \rangle |^2 
	=  | \langle v, R^* n_1 \rangle |^2 \\
	&= | \langle v, \bar{r}_2 n_1 \rangle |^2 
	=  | \langle v,  n_1 \rangle |^2 \\
	&= P_{E_1^{\perp}} (v, v) .	
	\end{align*}
	In the same way, \(R \cdot P_{E_2^{\perp}} = P_{E_2^{\perp}}\) and we conclude that the \(2\)-dimensional vector subspace of \(\R^{1,3}\) spanned by \(P_{E_1^{\perp}}\) and \(P_{E_2^{\perp}}\) is fixed by the action of \(R\). On the other hand, since \(\inn_H \mapsto R \cdot \inn_{H}\) is an orientation preserving linear isomorphism of \(\R^{1,4}\) different from the identity, it can not fix a subspace of dimension \( \geq 3\).
\end{proof}

\begin{lemma}\label{lem:hermform}
	Let \(R, S \in GL(2, \C)\) be as in Lemma \ref{lem:circle}. Then there is a unique up to scalar multiplication non-degenerate Hermitian form \(\inn_{H}\) on \(\C^2\) invariant under both \(R\) and \(S\). Moreover \(\inn_{H}\) is definite if and only if the two pairs of eigenvalues \(\{r_1, r_2\}\) and \(\{s_1, s_2\}\) interlace in the unit circle.
\end{lemma}

\begin{proof}
	Consider the projective space \(\RP^3 = \mathbb{P}(\R^{1,3})\) with homogeneous coordinates \([x_0, x_1, x_2, x_3]\).
	If \(L \subset \C^2\) is a complex line then \(P_{L^{\perp}}\) has coordinate
	\(x_0 = \tr P_{L^{\perp}} /2 = 1/2\) and its image in the chart
	\begin{equation}\label{eq:chart}
			[x_0, x_1, x_2, x_3] \mapsto \left( \frac{x_1}{x_0}, \frac{x_2}{x_0}, \frac{x_3}{x_0} \right)
	\end{equation}
	is equal to the point \(\bx(L) \in \R^3\) where \(\bx(L)\) is as in Lemma \ref{lem:sphere}. In particular, if \(L = \C \cdot (\lambda, 1)\) then
	\begin{equation}\label{eq:chart2}
		\bx(L) = \frac{1}{1+|\lambda|^2} (|\lambda|^2-1, 2 \re(\lambda), 2\im(\lambda))		
	\end{equation}
	is the usual stereographic projection (see Equation \eqref{eq:sterproj}).
	
	By Lemma \ref{lem:fixedH} the set of Hermitian forms preserved by \(R\) make a line \(\Fix(R) \subset \RP^3\) going through \(\bx(E_{r_1})\) and \(\bx(E_{r_2})\). By Lemma \ref{lem:circle} the lines \(\Fix(R)\) and \(\Fix(s)\) lie on a plane \(\RP^2 \subset \RP^3\) and so they intersect at a point, which represents the unique up to scale Hermitian form \(\inn_{H}\) fixed by \(R\) and \(S\). In order to compute the signature of \(\inn_{H}\) we
	can make a linear change of coordinates and assume that the matrices \(R, S\) are given by \eqref{eq:monmat} with eigenvectors
	\[
	(-r_2, 1), \hspace{2mm} (-r_1, 1), \hspace{2mm} (-s_2, 1), \hspace{2mm} (-s_1,1) .
	\]
	By Equation \eqref{eq:chart2} the projections to the orthogonal complements lie on the plane \(\{x_1=0\} \subset \RP^3\), their images in the chart \eqref{eq:chart} are on the circle \((x_2/x_0)^2 + (x_3/x_0)^2 =1\) and have coordinates
	\[
	\xi = (x_2/x_0)+i(x_3/x_0) = \hspace{2mm} -r_2, \hspace{2mm}, -r_1, \hspace{2mm} -s_1, \hspace{2mm} -s_2
	\]
	On the other hand the definite Hermitian forms \(\det H = x_0^2 - x_1^2 - x_2^2 - x_3^2 >0\) are mapped to the unit ball by \eqref{eq:chart}. The invariant Hermitian form \(\inn_{H}\) is definite if and only if
	the line \(\Fix(R)\) (that connects \(-r_2\) with \(-r_1\)) intersects \(\Fix(S)\) (that connects \(-s_2\) with \(-s_1\)) inside the unit disc \(\{|\xi|<1\}\) and this happens if and only if the \(2\) points \(-s_1, -s_2\) lie on opposite sides of the line \(\Fix(R)\) or equivalently the two pairs of points \(\{-r_1, -r_2\}\) and \(\{-s_1, -s_2\}\) interlace in the unit circle.
\end{proof}

\subsection{The case of \(3\) lines.}\label{sec:3linesHF}
Without loss of generality we let
\[
L_1= \{z=0\}, \hspace{2mm} L_2 = \{w=0\}, \hspace{2mm} L_3 = \{w=z\} .
\]
The unique standard connection with prescribed residue traces \(a_1, a_2, a_3\) is
\begin{equation}\label{eq:3lines}
\nabla = d -  \left( A_1 \frac{dz}{z} + A_2 \frac{dw}{w} + A_3 \frac{d(z-w)}{z-w} \right)
\end{equation}
with
\[
A_1 = \begin{pmatrix}
a_1 & 0 \\
\frac{a_1 + a_3 - a_2}{2} & 0
\end{pmatrix}, \hspace{2mm}
A_2 = \begin{pmatrix}
0 & \frac{a_2+a_3-a_1}{2} \\
0 & a_2
\end{pmatrix},
\hspace{2mm}
A_3=
\begin{pmatrix}
\frac{a_2+a_3-a_1}{2} & \frac{a_1-a_2-a_3}{2}\\
\frac{a_2-a_1-a_3}{2} & \frac{a_1+a_3-a_2}{2}
\end{pmatrix} .
\]

In this section we assume that \(a_i \in \R\). It is convenient to introduce parameters \(b_1, b_2, b_3\) such that 
\begin{equation}\label{eq:ab}
	a_i = b_j + b_k
\end{equation}
over cyclic permutations \((i, j, k)\) of \((1,2,3)\).
The main result is the next.

\begin{proposition}[Invariant Hermitian form]\label{prop:3lines}
	Let \(a_1, a_2, a_3\) be real and non-integer. Assume that \(b_i \notin \Z\) for \(i=1, 2, 3\) and that \(b_1 + b_2 + b_3 \notin \Z\). 
	
	Then \(\nabla\) preserves a non-degenerate Hermitian form whose signature \((p,2-p)\) is given by the following formula
	\begin{equation}\label{eq:signature}
	p = \lfloor s \rfloor \hspace{2mm} \text{ with } \hspace{2mm} s= \sum_i \{b_i\}
	\end{equation}
	where \(0 < \{b_i\} < 1\) denotes the \(\Z\)-periodic fractional part.
\end{proposition}

\begin{remark}
	The parameters \(b_i\) in Equation \eqref{eq:ab} are uniquely determined and given by
	\begin{equation}\label{eq:ba}
		b_i = \frac{a_j + a_k - a_i}{2} .
	\end{equation}
	In particular, we have that
	\begin{equation}\label{eq:sumb}
		b_1 + b_2 + b_3 = \frac{a_1+a_2+a_3}{2} .
	\end{equation}
	It follows from Equations \eqref{eq:ba} and \eqref{eq:sumb} that the two conditions together \(b_i \notin \Z\) and \(\sum_i b_i \notin \Z\) in Proposition \ref{prop:3lines} are equivalent to Condition \ref{noin2} in Corollary \ref{cor:irrhol0}. Therefore, in the setting of Proposition \ref{prop:3lines} the holonomy of \(\nabla\) is irreducible.
\end{remark}

\begin{proof}[Proof of Proposition \ref{prop:3lines}]
Fix \(x_0 \in \C^2 \setminus \cup_i L_i\) and take a basis of \(T_{x_0}\C^2\). By Lemma \ref{lem:mon} the holonomy of \(\nabla\) is a subgroup of \(GL(2, \C)\) given by 
\[
\Hol(\nabla) = \Bigl< M_1, M_2, M_3 \hspace{2mm} | \hspace{2mm} M_1 M_2 M_3 = \exp(2\pi i c) \cdot \Id \Bigr>
\]
where \(M_i\) has eigenvalues \(1, \exp(2\pi i a_i)\) and \(2c = \sum_i a_i\). Let  
\begin{equation}\label{eq:RS}
R = \exp(-2\pi i c) \cdot M_1, \hspace{2mm} S = M_2^{-1} .
\end{equation}
Then \(\Hol(\nabla)\) is generated by the \(3\) elements \(R, S, c \cdot \Id\) and
\begin{equation}\label{eq:relation}
RS^{-1} = M_3^{-1} .
\end{equation}
The advantage of this presentation for the holonomy group is that \( c \cdot \Id\) preserves all Hermitian forms (because \(|c|=1\)) and so \(\inn_{H}\) is invariant by \(\Hol(\nabla)\) if and only if it is invariant under both \(R, S\).

Equation \eqref{eq:RS} implies that the eigenvalues \(R, S\) are
\begin{equation*}
r_1 = \exp(-2\pi i c), \hspace{2mm} r_2 = \exp(2\pi i (a_1-c)), \hspace{2mm} s_1 = 1, \hspace{2mm} s_2=\exp(-2\pi i a_2) .
\end{equation*}
Note that \(r_1 \neq r_2\) because \(a_1 \notin \Z\) and \(s_1 \neq s_2\) because \(a_2 \notin \Z\). If we replace \(a_i = b_j + b_k\) then the arguments\footnote{Here we take the argument of a unit complex number \(\exp(2\pi i \theta)\) to be equal to \(\theta \in \R / \Z\).} of \(\bar{r}_1, \bar{r}_2, \bar{s}_1, \bar{s}_2\) are
\begin{equation}\label{eq:arguments}
b_1+b_2+b_3, \hspace{2mm} b_1, \hspace{2mm} 0, \hspace{2mm}  b_1+b_3 .
\end{equation}
The fact that \(\{r_1, r_2\} \cap \{s_1, s_2\} = \emptyset\) now follows from the assumption that \(b_i \notin \Z\) and \(\sum_i b_i \notin \Z\).

Equation \eqref{eq:relation} implies that \(RS^{-1}\) has eigenvalue \(1\). We apply Lemma \ref{lem:hermform} to get a non-degenerate invariant Hermitian form \(\inn_H\). Finally, Equation \eqref{eq:signature} for the signature follows from the next elementary Lemma \ref{lem:interlace}.	
\end{proof}

\begin{lemma}\label{lem:interlace}
	The \(2\) pair of points in the circle \(\{r_1, r_2\}\) and \(\{s_1, s_2\}\)  interlace if and only if \(\sum_i \{b_i\} < 1\) or \(\sum_i \{b_i\}>2\).
\end{lemma}

\begin{proof}
	Equation \eqref{eq:arguments} implies that the interlace condition holds if and only if
	\[
	\{b_1\} < \{b_1 + b_3\} < \{b_1 + b_2 + b_3\}, \hspace{2mm} \text{ or } \hspace{2mm} \{b_1+b_2+b_3\} < \{b_1 + b_3\} < \{b_1\} .
	\]
	Using the formula for the fractional part of the sum of two real numbers \(x, y\)
	\begin{equation*}
		\{x+y\} = \begin{cases}
		\{x\} + \{y\} & \text{ if } \{x\} + \{y\} < 1 \\
		\{x\} + \{y\} -1 & \text{ if } \{x\} + \{y\} \geq 1
		\end{cases}
	\end{equation*}
	it is easy to see that
	\[
		\{b_1\} < \{b_1 + b_3\} < \{b_1 + b_2 + b_3\} \iff \sum_i \{b_i\} < 1
	\]
	\[
	 \{b_1+b_2+b_3\} < \{b_1 + b_3\} < \{b_1\} \iff \sum_i \{b_i\}>2 .
	 \qedhere
	\]
\end{proof}

\begin{remark}
	Higher dimensional versions of Proposition \ref{prop:3lines} hold for the braid arrangement and Lauricella connection, see \cite[Section 3.6]{CHL} and \cite[Section 9]{dBP}.
\end{remark}

The next result characterizes the values of the real parameters \(a_1, a_2, a_3\)
for which the connection \eqref{eq:3lines} is Dunkl.
The next Lemmas \ref{lem:3dunkl} and \ref{lem:dunkunit} are not needed for the proof of Theorem \ref{thm:main} but we include them for completeness.

\begin{lemma}\label{lem:3dunkl}
	Let \(a_1, a_2, a_3 \in \R^*\). Then \(\nabla\) is Dunkl if and only if
	\begin{equation}\label{eq:moda}
	|a_i| < |a_j| + |a_k|
	\end{equation}
	for all \((i,j,k)\) running over cyclic permutations of \((1,2,3)\). 
\end{lemma}

\begin{proof}
	There is a unique up to scale Hermitian inner product \(\inn_H\) on \(\C^2\) that makes \(A_1, A_2, A_3\) self-adjoint, it is given by the Hermitian matrix
	\footnote{The matrix \(H\) is obtained by solving a linear system in \(4\) variables (the entries of \(H\)) and \(3\) equations (that the columns of \(A_i\) are orthogonal with respect to \(\inn_H\) to \(L_i\)).}
	\begin{equation*}
	H = \begin{pmatrix}
	b_2(b_1+b_3) & -b_1 b_2 \\
	-b_1 b_2 & b_1 (b_2+b_3)
	\end{pmatrix} .
	\end{equation*}
	The connection \(\nabla\) is Dunkl if and only if 
	\begin{equation*}
	\det H = b_1 b_2 b_3 (b_1 + b_2 + b_3) > 0 .
	\end{equation*}
	Up to a constant positive factor, \(\det H\) is equal to
	\[
	D = (a_2+a_3-a_1)(a_1+a_3-a_2)(a_1+a_2-a_3)(a_1+a_2+a_3) 
	\]
	which is invariant under permutations of \((a_1, a_2, a_3)\) and under sign changes \(a_i \mapsto -a_i\). In particular, \(D(a_1, a_2, a_3)=D(|a_1|, |a_2|, |a_3|)\) and Equation \eqref{eq:moda} is equivalent to \(D>0\).
\end{proof}

Finally we compare the Dunkl versus the unitary condition for \(\nabla\).
We show that the \(2\) conditions agree if \(a_i \in (0,1)\) and we are in the elliptic case \(\sum_i a_i < 2\). However, the two conditions are genuinely different for arbitrary values of the parameters \(a_i\) as can be seeing by comparing Equations \eqref{eq:signature} and \eqref{eq:moda}. 

\begin{lemma}\label{lem:dunkunit}
If \(0<a_i<1\) and \(a_1+a_2+a_3<2\) then \(\nabla\) is unitary  if and only if it is Dunkl. 
\end{lemma}

\begin{proof}
	Note that \(2b_i = a_j+a_k-a_i\). If one of the \(b_i\)'s is negative then the term \(s\) in Equation \eqref{eq:signature} is equal to
	\begin{equation}\label{eq:sb}
	s = b_1 + b_2 + b_3 + 1 .	
	\end{equation}
	On the other hand, the sum \(\sum_i b_i = (1/2) \sum_i a_i \) belongs to the interval \((0,1)\) and Equation \eqref{eq:sb} implies that \(s \in (1,2)\). Therefore, if \(b_i<0\) for some \(i\) then the signature of the Hermitian form preserved by \(\nabla\) is \((1,1)\). We conclude that \(\nabla\) is unitary if and only if \(b_i>0\) for all \(i\).
	
	On the other hand, by Lemma \ref{lem:3dunkl} (or Proposition \ref{prop:Dunkl}) the connection \(\nabla\) is Dunkl if and only if \(a_i < a_j + a_k\) for all \(i\)  which is also equivalent to \(b_i >0\) for all \(i\).
\end{proof}


\subsection{\(B_2\)-arrangement.}
Consider the reflection arrangement made of the \(4\) lines 
\[
L_1=\{z=0\}, \hspace{2mm} L_2 =\{w=0\}, \hspace{2mm} L_3=\{w=z\}, \hspace{2mm} L_4=\{w=-z\} .
\]
Let \(a \in \R^*\),
we are interested in Dunkl connections with \(\tr A_i = a\) for all \(i\). 

Since the configuration of points \(0, 1, -1, \infty\) in \(\CP^1\) has zero barycentre, the standard Euclidean inner product gives us the Dunkl connection
\begin{equation}\label{eq:dunkdihed}
\nabla_a = d - a \cdot \sum_i A_i \frac{d \ell_i}{\ell_i} \hspace{4mm} \text{ with } 
\end{equation}
\[
A_1 = \begin{pmatrix}
1 & 0 \\
0 & 0
\end{pmatrix},
\hspace{2mm}
A_2 = \begin{pmatrix}
0 & 0 \\
0 & 1
\end{pmatrix},
\hspace{2mm}
A_3 = \frac{1}{2} \begin{pmatrix}
1 & -1 \\
-1 & 1
\end{pmatrix},
\hspace{2mm}
A_4 = \frac{1}{2} \begin{pmatrix}
1 & 1 \\
1 & 1
\end{pmatrix} .
\]

The main result of this section is the following.

\begin{proposition}\label{prop:dihedral}
	Suppose that \(a \notin \Z\). Then the Dunkl connection \(\nabla_a\) has unitary holonomy if and only if either \(a \in (0,1/2) + 2\Z\) or \(a \in (3/2,2) + 2\Z\).
\end{proposition}

This proposition is an immediate consequence of the next two lemmas.

\begin{lemma}\label{lem:pullback}
	Let \(\tn\)  be the connection singular along \(3\) lines given by Equation \eqref{eq:3lines} with parameters 
	\begin{equation}\label{eq:ai}
		a_1 = a_2 = \frac{1+a}{2}, \hspace{2mm} a_3 = a .
	\end{equation}
	Then \(\nabla_a\) is equal to the pull-back of \(\tn\) by the map 
	 \begin{equation}
	 	F(z,w) = (z^2, w^2) .
	 \end{equation}
\end{lemma}

\begin{proof}
	Use linear coordinates \(x,y\) on the target so \(\tn = d - \widetilde{\Omega}\) with \(\widetilde{\Omega}\) equal to
	\[
	\frac{1}{2} \begin{pmatrix}
	1+a & 0 \\
	a & 0
	\end{pmatrix} \frac{dx}{x}
	+
	\frac{1}{2} \begin{pmatrix}
	0 & a \\
	0 & 1+a
	\end{pmatrix} \frac{dy}{y}
	+
	\frac{1}{2} \begin{pmatrix}
	a & -a \\
	-a & a
	\end{pmatrix} \frac{d(x-y)}{x-y} .
	\]
	The pull-back \(F^* \widetilde{\Omega}\) is obtained by replacing \(F^*(dx/x) = 2 dz/z\) (and similarly for the \(y\) variable) together with
	\[
	F^* \left( \frac{d(x-y)}{x-y}\right) =  \frac{d(z-w)}{z-w} + \frac{d(z+w)}{z+w} .
	\]
	The pull-back connection \(\nabla = F^* \tn\) in the frame \(\p_z, \p_w\) is
	\(\nabla = d - \Omega\) where
	\[
	\Omega = G \cdot F^*\widetilde{\Omega} \cdot G^{-1} - dG \cdot G^{-1} \hspace{2mm} \text{ with } \hspace{2mm} 	G = \begin{pmatrix}
	2z & 0 \\
	0 & 2w
	\end{pmatrix} 
	\]	
	(here the gauge transform \(G\) is the Jacobian \(DF\)).
	A straightforward check shows that \(\nabla\) is equal to the Dunkl connection \(\nabla_a\) given by Equation \eqref{eq:dunkdihed}.
\end{proof}

\begin{lemma}\label{lem:unitary}
	Suppose that \(a \notin \Z\). Then the connection \(\tn\) is unitary if and only if either \(a \in (0,1/2) + 2\Z\) or \(a \in (3/2,2) + 2\Z\).
\end{lemma}

\begin{proof}
	Let \(2b_i = a_j + a_k - a_i\) where \(a_i\) are given by Equation \eqref{eq:ai}. Then
	\[
	b_1=b_2=a/2, \hspace{2mm} b_3=1/2
	\]
	and the sum \(s=\sum_i \{b_i\}\) is equal to the \((2\Z)\)-periodic function
	\begin{equation}
		s = 2 \{a/2\} + 1/2 .
	\end{equation}
	If \(a \in (0,2)\) then \(s\) is \(<1\) (or \(>2\)) precisely when \(a \in (0,1/2)\) (or \(a \in (3/2,2)\)). The lemma now follows from Equation \eqref{eq:signature} for the signature of the Hermitian form preserved by \(\tn\).
\end{proof}

\begin{proof}[Proof of Proposition \ref{prop:dihedral}]
	By Lemma \ref{lem:pullback} we have \(\nabla_a = F^* \tn\). By Lemma \ref{lem:unitary} the Dunkl connection \(\nabla_a\) is unitary if and only if
	\(a \in (0,1/2) + 2\Z\) or \(a \in (3/2,2) + 2\Z\). 
\end{proof}

\section{Proof of Theorem \ref{thm:main}}\label{sec:pf}
In Section \ref{sec:precise} we state Theorem \ref{thm:main} in its precise version: \(Z\) is a proper analytic subset of \(\m \times \R\), where
\(Z\) is the set of pairs \((\lambda, a)\) such that the Dunkl connection \(\nd_{\lambda, a}\) (with residue traces \(=a\) and simple poles at \(\lambda\)) preserves a non-zero Hermitian form.
In Section \ref{sec:afam} we show that \(Z\) is an analytic subset of \(\m \times \R\). In Section \ref{sec:ops} we show that if \(Z = \m \times \R\) then, for suitable values of \(a\), the Dunkl connections \(\nd_{\lambda, a}\) have unitary holonomy for all \(\lambda \in \m\). 
Sections \ref{sec:k4} and \ref{sec:top} contain preliminary material on the Klein four-group and elementary topology.
The contradiction argument that proves the main theorem is carried out in Section \ref{sec:pfthmp}. Finally, in Section \ref{sec:extension} we extend our results to the general case of any number lines and possible different residue traces.

\subsection{Precise statement.}\label{sec:precise}
Let \(\m = \CP^1 \setminus \{0,1,\infty\}\). To each \(\lambda \in \m\) we associate the configuration of \(4\) lines
\[
L_1 = \{z=0\}, \hspace{1mm} L_2 = \{w=0\}, \hspace{1mm} L_3 = \{w=z\}, \hspace{1mm} L_4 = \{z=\lambda w\} .
\]

Given \(\lambda \in \m\) we denote by \(H(\lambda)\) the unique positive definite Hermitian matrix of unit determinant such that \(\inn_{H(\lambda)}\) is a Dunkl inner product adapted to \((L_i, a_i=1)\) (as given by Proposition \ref{prop:Dunkl}).
For \(a \in \R\) and \(\lambda \in \m\) we let \(\nd_{\lambda, a}\) be the Dunkl connection 
\begin{equation}\label{eq:dunkfamily}
\nd_{\lambda, a} = d - a \sum_i A_i(\lambda) \frac{d\ell_i}{\ell_i}
\end{equation}
where \(A_i(\lambda)\) is the orthogonal projection to \(L_i^{\perp}\) with respect to \({H(\lambda)}\). 

The precise statement of Theorem \ref{thm:main} is the following.
\begin{manualtheorem}{1'}\label{thm:main2}
	Let \(Z\) be the set of pairs \((\lambda, a)\) such that \(\nd_{\lambda, a}\) preserves a non-zero Hermitian form. Then \(Z\) is a proper analytic subset of \(\m \times \R\). In other words, \(Z\) is locally the zero set of real analytic functions which are not identically zero. In particular, the set \(Z\) is closed and has measure zero.
\end{manualtheorem}

\begin{remark}
	Note that \( \m \times \{0\} \subset Z\) because when \(a=0\) then \(\nd_{\lambda, a}\) is the usual Levi-Civita connection of the flat metric on \(\C^2\). Remark \ref{rmk:flatherm} also implies that \((\lambda, a) \in Z\) for all \(\lambda\) whenever \(a=1/2\). 
\end{remark}

\subsection{Analytic families of connections.}\label{sec:afam}
The main result here is Lemma \ref{lem:locan} which shows that \(Z\) is an analytic subset. Its proof relies on the results from Section \ref{sec:Dunkl}.
We begin with a general elementary fact about persistence of fixed subspaces by analytic families of representations.

\begin{lemma}\label{genericparallel} Let $G$ be a finitely generated group and let $\rho_t$ for \(t \in \R\)  be a real analytic family of representations $\rho_t: G\to GL(n, \R)$. Suppose that for $t\in (0,1)$ the subspace $F_t\subset \R^n$ fixed by $\rho_t$ is non-zero. Then it is non-zero for all $t$.
\end{lemma}
\begin{proof} Let $g_1,\ldots, g_k$ be the generators of $G$. Then for any $t$ the space $F_t$ fixed by $\rho_t$ coincides with the kernel of the following non-negative definite self-adjoint operator:
	\begin{equation}\label{eq:Qt}
		Q_t=\sum_i(\rho_t(g_i)^*-1)(\rho_t(g_i)-1).
	\end{equation}
	Since $\det(Q_t)=0$ for $t\in (0,1)$ and $Q_t$ is real analytic, we conclude $\det(Q_t)=0$ for all $t$. So $F_t$ is non-zero for all $t$.
\end{proof}

\begin{corollary} \label{cor:persistentHermitian}
	Consider an analytic path of standard connections 
	\[
	\nabla_t = d - \sum_i A_i(t) \frac{d\ell_i}{\ell_i}
	\]	
	meaning that the entries of the matrices \(A_i(t)\) are analytic functions of \(t\).	Then the subset of $t\in \R$ such that the connection $\nabla_t$ preserves a non-zero Hermitian form is either $\mathbb R$ or a discrete subset of $\mathbb R$.
\end{corollary}

\begin{proof}
	Let \(G\) be the fundamental group of the arrangement complement.
	We have an analytic family of holonomy representations \(\tilde{\rho}_t: G \to GL(2, \C)\). Recall that \(A \in GL(2, \C)\) acts on the space of Hermitian forms in \(\C^2\) by pull-back \(A \cdot \inn_H = \langle A \cdot, A \cdot \rangle_H\) giving rise to a homomorphism \(GL(2, \C) \to GL(4, \R)\). By composition we obtain an analytic family of representations \(\rho_t : G \to GL(4, \R)\) and \(\nabla_t\) preserves a non-zero Hermitian form if and only if there is a non-zero subspace \(F_t \subset \R^4\) fixed by \(\rho_t\).
\end{proof}

\begin{remark}
	A simple but important observation is that if \(\nabla = d - \sum_i A_i d\ell_i / \ell_i\) is a standard (or Dunkl) connection then 
	\begin{equation}\label{eq:scalepath}
	\nabla_t = d - t \cdot \sum_i  A_i \frac{d\ell_i}{\ell_i}	
	\end{equation}
	is also standard (or Dunkl)\footnote{Note that the residue traces of \(\nabla_t\) are \(t a_i\) where \(a_i =  \tr A_i\).}  for every \(t \in \R\). In particular, Corollary \ref{cor:persistentHermitian} applies here. This family of connections \(\nabla_t\) is studied in \cite{CHL} in the Dunkl case for reflection arrangements.
\end{remark}

Now we go back to our family of Dunkl connections given by Equation \eqref{eq:dunkfamily}.

\begin{lemma}\label{lem:analyticdependence}
	The entries of  \(A_i(\lambda)\) are analytic functions of \(\lambda \in \m\).
\end{lemma}

\begin{proof}
	The matrix \(A_i(\lambda)\) is the orthogonal projection to \(L_i^{\perp}\) with respect to the Dunkl inner product \(H(\lambda) \in \h^3\), so it is enough to show that \(H(\lambda)\) depends analytically on \(\lambda\).
	
	The Busemann functions \(b_x(y)\) are analytic\footnote{This is evident from Equation \eqref{eq:buseman}.} in \((x, y) \in S^2 \times \h^3\).
	The point \(H(\lambda) \in \h^3\) is the unique critical point of \(F(\lambda, \cdot): \h^3 \to \R\) given by
	\[
	F(\lambda, y) = b_S(y) + b_E(y) + b_N(y) + b_{\bx(\lambda)} (y)
	\]
	where \(S = (-1,0,0)\), \(E=(0,1,0)\), \(N = (1, 0, 0)\) are the points in \(S^2\) corresponding (as in Lemma \ref{lem:sphere}) to the lines \(L_1, L_2, L_3\) and \(\bx(\lambda)\) is the image of \(L_4\) given by stereographic projection (Equation \eqref{eq:sterproj}) which is analytic. Therefore \(F\) is analytic in \((\lambda, y) \in \m \times \h^3\). The point \(H(\lambda)\) is the unique solution of \(\nabla_y F (\lambda, H(\lambda))=0\). By Lemma \ref{lem:convex} the Hessian of \(F\) is non-degenerate and the statement follows from the implicit function theorem for analytic functions.
\end{proof}

As a consequence of analyticity and the persistence principle of Lemma \ref{genericparallel} we have the following.

\begin{lemma}\label{lem:allornoth}
	If \(\nd_{\lambda, a}\) preserves a non-zero Hermitian form for all \(\lambda\) in some open set \(U \subset \m\) and all \(a\) in some open interval \(I \subset \R\). Then \(\nd_{\lambda, a}\) preserves a non-zero Hermitian form for \emph{all} \(\lambda \in \m\) and \emph{all} \(a \in \R\). 
\end{lemma}

\begin{proof}
	Let \((\lambda, a) \in \m \times \R\). Take an analytic path \(\R \to \m \times \R\) that contains \((\lambda, a)\) in its interior and intersects \(U\times I\). The statement now follows by from Corollary \ref{cor:persistentHermitian}.
\end{proof}

Finally, we adapt the argument in Lemma \ref{genericparallel} to obtain the following.

\begin{lemma}\label{lem:locan}
	Let \(\lambda_0 \in \m\). Then we can find an open neighbourhood \(U \subset \m\) of \(\lambda_0\) such that 
	\[
	Z \cap (U \times \R) = \{f=0\}
	\]
	where \(f: U \times \R \to \R\) is a real analytic function. I.e., \(Z\) is an analytic subset.
\end{lemma}

\begin{proof}
	By Lemma \ref{lem:analyticdependence} the family of Dunkl connections \(\nd_{\lambda, a}\) given by Equation \eqref{eq:dunkfamily} depends analytically on \((\lambda, a)\). 
	Take generators \(g_1, \ldots, g_k\) of the fundamental group \(G=\pi_1(\C^2 \setminus \cup_i L_i, x_0)\) where \(x_0\) is a fixed point outside the lines and take a basis of \(T_{x_0}\C^2\). The holonomy of the Dunkl connections \(\nd_{\lambda, a}\) make an analytic family of representations \(\rho_{\lambda, a}: G \to GL(\R^4)\) where \(\R^4\) is the space of Hermitian inner products in \(T_{x_0}\C^2\). Set
	\begin{equation}\label{eq:Q}
	Q(\lambda, a) = \sum_i (\rho_{\lambda,a}(g_i)^*-1)(\rho_{\lambda,a}(g_i)-1).	
	\end{equation}
	Note that \(Q\) is a non-negative operator with kernel equal to the invariant Hermitian forms. The function \(f = \det Q\) is analytic and \(Z=\{f=0\}\).
\end{proof}

\subsection{One parameter families of flat Hermitian forms.}\label{sec:ops}
The main result of this section is Lemma \ref{lem:posdef}. It says that if \(\nd_{\lambda, a}\) preserves a non-zero Hermitian form \(h_{\lambda}\) for all \(\lambda \in \m\) and some fixed \(a \in ((0,1/2) \cup (3/2, 2)) + 2 \Z\) then
\(h_{\lambda}\) must be definite. The proof combines the results from Sections \ref{sec:holonomy} and \ref{sec:Dunkunitary}. We use that \(h_{\lambda}\) is definite when \(\lambda=-1\) (dihedral arrangement). On the other hand, if we have a path \(\lambda(t)\) starting at \(\lambda(0)=-1\) then the Hermitian forms \(h_{\lambda(t)}\) must all be definite as follows from the fact that the holonomy of the connection \(\nd_{\lambda, a}\) is irreducible for all \(\lambda\). We begin with a standard fact about continuous families of linear operators.

\begin{lemma}\label{lem:kermatrices}
	Let \(B\) be a topological space. Suppose that for each \(\lambda \in B\) we have a matrix \(P_{\lambda} \in M(n \times n, \R)\) whose entries depend continuously on \(\lambda\) and such that \(\ker P_{\lambda}\) is  \(1\)-dimensional for all \(\lambda\). Then we can locally find a continuous family of non-zero vectors \(\lambda \mapsto v_{\lambda} \in \R^n\) such that \(P_{\lambda}v_{\lambda}=0\).
\end{lemma}

\begin{proof}
	Since the projection map \(\R^{n+1} \setminus \{0\} \to \RP^{n-1}\) is a locally trivial fibration, it is enough to show that the map \(\lambda \mapsto \ker P_{\lambda} \in \RP^{n-1}\) is continuous. If not then we would have some \(\lambda_0\) and a sequence \(\lambda_i \to \lambda_0\) as \(i \to \infty\) such that \(p_i = \ker P_{\lambda_i}\) belong to \(\RP^{n-1} \setminus U\) where \(U\) is some neighbourhood of \(p_0 = \ker P_{\lambda_0}\). By compactness we can take a subsequence such that \(p_i \to p_{\infty}\) where \(p_{\infty} \in \RP^{n-1} \setminus U\). But then \( \ker P_{\lambda_0}\) would be at least two dimensional, contradicting the assumption.
\end{proof}

\begin{lemma}\label{lem:posdef}
	Fix \(a \in I\)  where \(I\) is the disjoint collection of intervals of the form 
	\begin{equation}
	\left(0,\frac{1}{2}\right) + 2\Z \hspace{2mm} \text{ or } \hspace{2mm} \left(\frac{3}{2},2 \right) + 2\Z 	.
	\end{equation}
	Suppose that \(\nd_{\lambda, a}\) preserves a non-zero Hermitian form \(h_{\lambda}\) form for all \(\lambda \in \m\) then the following holds.
	\begin{enumerate}
		\item[(i)] Up to scale, \(h_{\lambda}\) is the unique Hermitian form preserved by \(\nd_{\lambda, a}\) and \(h_{\lambda}\) is non-degenerate.
		\item[(ii)] We can take \(h_{\lambda}>0\) that depend continuously on \(\lambda \in \m\). In the sense that, for \(\lambda\) varying in a small open set \(U \subset \m\) and for each point \(x_0\) outside the collection of four lines, we have a continuous family of Hermitian inner products in the tangent space \(T_{x_0}\C^2\).
	\end{enumerate}
\end{lemma}

\begin{proof}
	\((i)\) If \(h'_{\lambda}\) was another flat Hermitian form then, for a suitable choice of \(s \in \R\) the flat Hermitian form \(h_{\lambda} -s h'_{\lambda}\) would have non-zero kernel. If this kernel is \(1\)-dimensional then we get an invariant rank \(1\) holomorphic foliation. 
	On the other hand, if \(a \in I\) then the non-integer assumptions of Proposition \ref{prop:flatbundle} are satisfied and there are no rank \(1\) invariant foliations. Therefore the above kernel has full rank and \(h_{\lambda} = s h'_{\lambda}\).
	
	\((ii)\) The flat connections \(\nd_{\lambda, a}\) produce a continuous family of representations \(\rho_{\lambda}: G \to GL(\R^4)\) where \(G\) is the fundamental group of the arrangement complement, \(\R^4\) is the vector space of Hermitian inner products on \(T_{x_0} \C^2\) and \(x_0 \in \C^2 \setminus \cup_i L_i\) is some fixed point.\footnote{Here we think of the parameter \(\lambda\) to be varying in a small open set \(U \subset \m\) so we can certainly take a point \(x_0\) that doesn't belong to any of the lines
	\(L_1 = \{z=0\}\), \(L_2 = \{w=0\}\), \(L_3 = \{w=z\}\), and \(L_4 = \{z=\lambda w\}\) for \(\lambda \in U\).}	
	As in Lemma \ref{genericparallel} we can define a continuous family of non-negative linear operators \(Q_{\lambda}: \R^4 \to \R^4\) such that \(\ker Q_{\lambda}\) are the flat Hermitian forms. By item \((i)\) \(\dim \ker Q_{\lambda} = 1\) for all \(\lambda\). Take a path \(\lambda(t)\) that connects some given \(\lambda = \lambda(1) \in \m\) to \(\lambda(0)=-1\) (dihedral arrangement).
	By Lemma \ref{lem:kermatrices} we can find a continuous family of non-zero Hermitian inner products \(H(t) \in \ker Q_{\lambda(t)}\). By Proposition \ref{prop:dihedral} we can assume (after multiplication by \(-1\) if necessary) that \(H(0) >0\). By item \((i)\) no \(H(t)\) has kernel, so we must have that \(H(t)\) is positive for all \(t\) and therefore \(h_{\lambda}>0\). 	
\end{proof}

\subsection{Klein four-group.}\label{sec:k4}
Given a set \(S\) of four points in the Riemann sphere there is a subgroup \(K \cong \Z_{2} \times \Z_{2}\) of M\"obius transformations that preserves \(S\). The quotient map of \(\CP^1\) by the \(K\)-action is a degree \(4\) branched covering map \(\Phi: \CP^1 \xrightarrow{4:1} \CP^1\). The image of \(S\) together with the \(3\) critical values of \(\Phi\) make another set \(S'\) of four points. The main result of this section is Lemma \ref{lem:k4}, it says that the configuration \(S'\) is equal to \(S\) by a M\"obius map. 

\begin{remark}
	An algebraic (shorter) proof of this fact is as follows. Suppose we have an elliptic curve \(E\) and consider the action on it by a group \(G\) of order \(8\) generated by \(3\) elements \(\sigma: x \mapsto -x\) and two translations by half periods. Then the statement of Lemma \ref{lem:k4} is that the quotients \(E/G\) and \(E/\sigma\) are isomorphic as orbifolds. We take the long route of straightforward computation but which has the advantage of fixing the marking convention for the \(4\) points to be used later in Section \ref{sec:pfthmp}.
\end{remark}

We begin by recalling the identification of \(\m\) with a trice punctured sphere via cross-ratio.
\(\m\) is the set of ordered \(4\)-tuples \((x_1,x_2,x_3, x_4)\) of points in  \(\CP^1\) up to the action of \(PSL(2,\C)\) by M\"obius transformations. There is a unique M\"obius map \(\mu\) that takes \(x_1, x_2, x_3\) to \(0,1, \infty\) and so a unique representative of the form \((0,1,\infty, \lambda)\) in the \(PSL(2, \C)\) orbit of \((x_1,x_2,x_3, x_4)\). 
More explicitly,
\begin{equation}\label{eq:crossratio}
\mu = \left(\frac{x_2-x_3}{x_2-x_1}\right) \left(\frac{z-x_1}{z-x_3}\right)  \hspace{2mm} \text{ and } \hspace{2mm} \lambda = \left(\frac{x_2-x_3}{x_2-x_1}\right) \left(\frac{x_4-x_1}{x_4-x_3}\right) 	
\end{equation}
is the cross-ratio of \((x_1, x_2, x_3, x_4)\). This way we identify \(\m = \CP^1 \setminus \{0,1,\infty\}\).

Consider now the action of the symmetric group \(S_4\) by permutations.
There are \(3\) ways to partition \(4\) objects into two pairs. This gives a homomorphism \(S_4 \to S_3\) with kernel the Klein \(4\)-group \(V = \{1, \sigma_1, \sigma_2, \sigma_3\}\) where
\[
\sigma_1 = (14)(23), \hspace{2mm} \sigma_2=(24)(13), \hspace{2mm} \sigma_3=(34)(12) .
\]
Easy inspection of Equation \eqref{eq:crossratio} shows that the \(\sigma_i\) leave invariant the cross-ratio \(\lambda\). In particular, we have the following.


\begin{lemma}\label{lem:Kgroup}
	Let \((x_1, x_2, x_3, x_4)\) be an ordered \(4\)-tuple of points in \(\CP^1\). Then there is a uniquely determined subgroup of M\"obius transformations 
	\[K = \{1, M_1, M_2, M_3\} \subset PSL(2, \C)\]  
	that preserves the set of \(4\) points extending the action of the Klein \(4\)-group on \((x_1, x_2, x_3, x_4)\) by permutations in which \(M_i\) restricts to \(\sigma_i\).
\end{lemma}

\begin{proof}
	We can assume that \((x_1, x_2, x_3, x_4) = (0,1,\infty, \lambda)\). Then
	\begin{equation}\label{eq:Mi}
	M_1 = \frac{z-\lambda}{z-1}, \hspace{2mm} M_2 = \frac{\lambda}{z}, \hspace{2mm} M_3 = \lambda \frac{z-1}{z-\lambda} . \qedhere
	\end{equation}	
\end{proof}

\begin{remark}\label{rmk:label}
	The Klein \(4\)-group acts simply transitively on \(1,2,3,4\). Any coset \([\sigma] \in S_4/V\) has a unique representative that fixes the element \(4\) thus giving an identification \(S_4/V = S_3\). This way a re-labelling  of the points \(x_{\sigma(i)}\) produces a re-labelling of the M\"obius maps \(M_{[\sigma](i)}\). 
	
	On the other hand, a labelling of the elements of \(V\) (or \(M\)) is equivalent to giving a group homomorphism \(V \cong \Z_{2}\times \Z_{2}\) say by \(\sigma_1 \mapsto (1,0)\) and \(\sigma_2 \mapsto (0,1)\).
\end{remark}

\begin{lemma}\label{lem:branchedcovering}
	There is a \(K\)-invariant degree \(4\) branched covering map
	\[
	\Phi: \CP^1 \to \CP^1
	\]
	that sends all the four points \(\{x_1, x_2, x_3, x_4\}\) to a regular value \(y_4\) and has \(3\) critical values \(y_1, y_2, y_3\) where \(y_i\) is the image of the fixed point set of \(M_i\). The map \(\Phi\) is uniquely determined up to \(PSL(2, \C)\) post-composition. 
\end{lemma}

\begin{proof}
	Take \((x_1, x_2, x_3, x_4) = (0,1,\infty, \lambda)\) so the M\"obius maps \(M_i\) are given by Equation \eqref{eq:Mi}.
	Each \(M_1, M_2, M_3\) has a pair of fixed points given by the roots of the polynomials 
	\[
	z^2 - 2z + \lambda, \hspace{2mm} z^2-\lambda, \hspace{2mm} z^2 - 2\lambda z + \lambda .
	\]
	To fix \(\Phi\) we require that \(y_1, y_2, y_3 = 0,1,\infty\).	
	This allows us to determine the unique map
	\begin{equation}\label{eq:4cover}
	\Phi(z) = \lambda \frac{(z^2-2z+\lambda)^2}{(z^2-2\lambda z+\lambda)^2}
	\end{equation}
	with the desired properties (here \(y_4=\lambda\)).
\end{proof}

By Lemma \ref{lem:branchedcovering} we obtain a well-defined map \(\m \to \m\) by
\begin{equation}\label{eq:4map}
(x_1, x_2, x_3, x_4) \mapsto (y_1, y_2, y_3, y_4) .
\end{equation}

Conversely, by the Riemann Existence Theorem \cite[Chapter 4.2.2]{don} if we are given \((y_1, y_2, y_3, y_4)\) then there is a (unique up to \(PSL(2, \C)\) pre-composition) degree \(4\) branched covering map \(\Phi\) given by the transitive permutation representation
\[
\pi_1 (\CP^1\setminus\{y_1, y_2, y_3\}) \to S_4
\]
that sends a positive loop encircling \(y_i\) to \(\sigma_i\). The labelling of the deck transformations corresponding to \(\sigma_1, \sigma_2, \sigma_3\)
determines a labelling for the elements of the fibre \(\Phi^{-1}(y_4) = \{x_1, x_2, x_3, x_4\}\) up to the action of the Klein \(4\)-group (see Remark \ref{rmk:label}). However, by Lemma \ref{lem:Kgroup} the \(4\) possible labellings of the points in \(\Phi^{-1}(y_4)\) define \(PSL(2, \C)\)-equivalent configurations. Thus we obtain a well-defined map \((y_1, y_2, y_3, y_4) \mapsto (x_1, x_2, x_3, x_4)\) that inverts \eqref{eq:4map}.

\begin{lemma}\label{lem:k4}
	The map \eqref{eq:4map} defines a holomorphic bijection of \(\m\). Indeed, with our choices, the ordered \(4\)-tuple \((y_1, y_2, y_3, y_4)\) is \(PSL(2, \C)\) equivalent to \((x_1, x_2, x_3, x_4)\). So the map \eqref{eq:4map} is simply the identity.
\end{lemma}

\begin{proof}
	Identify \(\m = \CP^1 \setminus \{0,1,\infty\}\) and let \(\Phi\) be given by Equation \eqref{eq:4cover}. Then the map \eqref{eq:4map} is  the identity transformation
	\begin{equation}
	\lambda \mapsto \Phi(\lambda) = \lambda . \qedhere
	\end{equation}
\end{proof}

\begin{remark}[Geometric interpretation.]
	Take a tetrahedron \(T\) with total angle \(\pi\) at each of its \(4\) vertices,
	so that its four faces are made of isometric triangles.\footnote{A tetrahedron \(T\) like this is known as disphenoid or isosceles tetrahedron.} A \(180^{\circ}\)-rotation around an axis that connects midpoints of opposite edges defines an isometry of \(T\). There are \(3\) such isometries, taking the quotient of \(T\) by them we get another tetrahedron \(T'\) which is a copy of \(T\) scaled by a factor of \(1/2\).
\end{remark}

\subsection{Elementary topology.}\label{sec:top}
Here we prove two easy results in topology needed for our arguments. These are the next Lemmas \ref{lem:topology} and \ref{lem:punctures}.

Recall that surfaces of genus \(\geq 1\) have universal cover homeomorphic to \(\R^2\), therefore all their higher homotopy groups \(\pi_i\) vanish for \(i \geq 2\). As a consequence, we have the next.

\begin{lemma}\label{lem:topology}
	Let \(\Sigma_g\) be a surface of genus \(g \geq 1\) and let \(F: \Sigma_g \to S^2\) be a continuous map. Then there is no continuous map \(s: S^2 \to \Sigma_g\) such that \(F \circ s = 1_{S^2}\).
\end{lemma}

\begin{proof}
	Suppose that such a map \(s\) exists. Then, since \(\pi_2(\Sigma_g)=0\), the map \(s\) is contractible and so is the composition \(F \circ s\). This is a contradiction because the identity map of the \(2\)-sphere is not contractible.
\end{proof}

The next result will be used in the proof of Theorem \ref{thm:main} to extend the forgetful map and the section over the punctures.

\begin{lemma}\label{lem:punctures}
	Let \(M\) and \(N\) be two compact surfaces and let \(P \subset M\) and \(Q \subset N\) be two finite subsets of points. Suppose that \(F: M \setminus P \to N \setminus Q\) is a continuous and proper map. Then \(F\) extends continuously as a map \(F: M \to N\) with \(F(P) \subset Q\).
\end{lemma}

\begin{proof}
	Let \(x_i\) be a sequence of points in \(M \setminus P\) that converges to a point \(p\) in \(P\). We want to show that the sequence \(y_i = F(x_i)\) converges to a point in \(Q\). Equivalently, if we denote by \(Y \subset N\) the set of accumulation points of the sequence \(y_i\), then we want to show two things: \((i)\) \(Y \subset Q\) and \((ii)\) \(|Y|=1\).
	
	\((i)\) We prove this by contradiction. If there was a point \(y_{\infty} \in Y \setminus Q\) then after taking a subsequence we would have \(y_{\infty} = \lim_{i \to \infty} y_i\). Consider the compact set \(K= \{ y_{\infty} \} \cup_i \{y_i\} \subset N \setminus Q\). Then the pre-image \(F^{-1}(K) \subset M \setminus P\) is non-compact because it contains the sequence \(x_i\), but this contradicts the properness assumption.
	
	\((ii)\) Note that \(Y\) is non-empty because \(N\) is compact, so we need to show that \(|Y| \leq 1\). We argue by contradiction again and suppose that there are \(2\) distinct points \(q_1, q_2 \in Y\). Let \(x_{1i}\) and \(x_{2i}\) be subsequences of \(\{x_i\}_{i=1}^{\infty}\) such that \(F(x_{1i})\) converges to \(q_1\) and \(F(x_{2i})\) converges to \(q_2\). By taking a concatenation of paths in \(M\) we can construct a continuous path \(c: (0,1) \to M \setminus P\) with \(\lim_{t \to 1} c(t)=p\) and a pair of interlacing sequences \(0< s_1 < t_1 < s_2 < t_2 < \ldots < 1\) with \(\lim_{k \to \infty} s_k = \lim_{k \to \infty} t_k =1\) such that \(c(s_k) = x_{1k}\) and \(c(t_k)=x_{2k}\) for all \(k\). Let \(D \subset N\) be a small disc about \(q_1\).  We can assume that the points \(F(c(s_k))\) belong to the interior of \(D\) while the points \(F(c(t_k))\) are outside the closure of \(D\). Then we can find a sequence of numbers \(s_k < r_k < t_k\) such that \(F(c(r_k))\) belongs to the boundary circle \(\p D \subset N \setminus Q\). Now the sequence of points \(z_k = c(r_k) \in M \setminus P\) converges to \(p\) as \(k \to \infty\) but the images \(F(z_k)\) have a limit point in \(N \setminus Q\) because the circle is compact. This contradicts \((i)\).
\end{proof}

\subsection{Proof of Theorem \ref{thm:main2}.}\label{sec:pfthmp}
By Lemma \ref{lem:locan}  \(Z\) is locally the zero-set of an analytic function \(Z=\{f=0\}\). It only remains to show that \(f\) does not vanish identically. 

\vspace{2mm}
\noindent\textbf{Contradiction assumption:} suppose that \(f \equiv 0\), or equivalently by Lemma \ref{lem:allornoth}, suppose that \(Z= \m \times \R\). 
\vspace{2mm}

Fix \(a<-(3/2)\) (this is \(\alpha>5/2\)) in one of the intervals \((1/2,1)+ 2\Z\) or \((1,3/2)+2\Z\). By Lemma \ref{lem:posdef} for each \(\lambda \in \m\) we can choose a positive definite Hermitian form \(h_{\lambda}>0\) which is preserved by \(\nd_{\lambda, a}\) and such that the family \(\lambda \mapsto h_{\lambda}\) is continuous. By the results in Appendix \ref{sec:unitaryconnect} we obtain a family of spherical metrics \(\lambda \mapsto g_{\lambda}\) that define a continuous map
\[
\m \to \MS_{0,4}(\alpha, \alpha, \alpha, \alpha) .
\]
\begin{lemma}\label{lem:invariance}
	Each \(g_{\lambda}\) is invariant under the M\"obius maps \(M_i\) given by Equation \eqref{eq:Mi} and so it can be pushed down by the map \eqref{eq:4cover} to a metric \(\tilde{g}_{\lambda}\) with cone angles \(\pi\) at \(0,1,\infty\) and \(2\pi\alpha\) at \(\lambda\).
\end{lemma}
\begin{proof}
	By the uniqueness of Dunkl connections proved in Proposition \ref{prop:Dunkl} the \(\nd_{\lambda, a}\) are invariant under the lifts of \(M_i\) to \(\C^2\). On the other hand, by Proposition \ref{prop:flatbundle} the holonomy of \(\nd_{\lambda, a}\) is irreducible and therefore the Hermitian form \(h_{\lambda}\) must be preserved up to a factor. Taking quotient, the spherical metrics are invariant \(M_i^* \tilde{g}_{\lambda}=\tilde{g}_{\lambda}\).
\end{proof}

Take the elliptic curve \(C_{\lambda}\) that branches over \(0,1, \infty, \lambda\) and pull-back the metric \(\tilde{g}_{\lambda}\) to obtain a spherical torus \(\hat{g}_{\lambda}\). The family \(s(\lambda)=\hat{g}_{\lambda}\) defines a continuous map 
\begin{equation}
	s: \m \to \MSmm .
\end{equation}
By Lemma \ref{lem:k4} we have
\begin{equation}
	F \circ s = 1_{\m}
\end{equation}
where \(F\) is the forgetful map defined in Appendix \ref{ap:ms}. By Theorem \ref{EMP} the space \(\MSmm\) is homeomorphic to a punctured (compact, orientable) surface \(\Sigma_g\) of genus \(g = (m-1)(m-2)/2\) where \(m=\lfloor (2\alpha+1)/2 \rfloor\). Our choice of \(\alpha>5/2\) guarantees that \(m \geq 3\) and therefore the genus of \(\Sigma_g\) is \(\geq 1\).

\begin{lemma}\label{lem:sproper}
	The map \(s\) is proper.
\end{lemma}

\begin{proof}
	Take a sequence of points \(\lambda_i\) in \(\m\) that converges to either \(0, 1\) or \(\infty\) and write \(g_i\) for the spherical tori \(s(\lambda_i)\). Since the conformal structure of \(s(\lambda_i)\) degenerates, the extremal systole tends to \(0\) as \(i \to \infty\). By \cite[Corollary A.10]{MPII} the systole \(\text{sys}(g_i)\)  of the spherical tori converges to \(0\). On the other hand, by \cite[Corollary 6.26]{EPM} the function \(1/\text{sys}\) is proper on \(\MSmm\) and therefore the sequence \(g_i\) must diverge to infinity in Lipschitz topology (i.e. converge to punctures of \(\MSmm\)).
\end{proof}

Let \(P \subset \Sigma_g\) be the finite set of punctures so that \(\MSmm\) is homeomorphic to \(\Sigma_g \setminus S\).

\begin{lemma}
	We can extend the forgetful map \(F\) and the section \(s\) both continuously over the punctures. More precisely, the following holds.
	\begin{itemize}
		\item[(i)] The map \(s\) extends over \(0, 1, \infty\) to a continuous map \(s: S^2 \to \Sigma_g\) by sending the points \(0, 1, \infty\) to points in \(P\).
		\item[(ii)] The map \(F\) extends over \(P\) to a continuous map \(F: \Sigma_g \to S^2\) by sending the points in \(P\) to \(0, 1, \infty\).
	\end{itemize}
\end{lemma}

\begin{proof}
	The two items follow from our general topology Lemma \ref{lem:punctures} once we show that \(s\) and \(F\) are proper. The properness of \(s\) follows from Lemma \ref{lem:sproper} while the properness of \(F\) follows from Theorem \ref{MP}.
\end{proof}

\noindent\textbf{Proof of Theorem \ref{thm:main2}.}
We have  produced a pair continuous maps \(F: \Sigma_g \to S^2\) and \(s: S^2 \to \Sigma_g\) such that \(F \circ s = 1_{S^2}\) but this contradicts Lemma \ref{lem:topology}.
\qed

\subsection{Extension to several lines.}\label{sec:extension}
Suppose that \(n \geq 4\) and let \(\mn\) be the configuration space of ordered \(n\)-tuples of complex lines in \(\C^2\) going through the origin up to linear equivalence. 

Let \(\lambda = (\lambda_1, \ldots, \lambda_{n-3}) \in (\C \setminus \{0,1\})^{n-3}\) be such that \(\lambda_i \neq \lambda_j\}\) for \(i \neq j\). Given \(\lambda\) we associate to it the configuration of \(n\) lines
\begin{equation}
L_1=\{z=0\}, \hspace{1mm} L_2=\{z=w\}, \hspace{1mm}  L_3=\{w=0\}, \hspace{1mm}
L_{i+3} = \{z=\lambda_i w\} 
\end{equation}
for \(	1 \leq i \leq n-3\). This way we can identify the configuration space \(\mn\) with the hyperplane arrangement complement
\[
\left(\C \setminus \{0,1\}\right)^{n-3} \setminus \left( \cup_{i \neq j} \{\lambda_i = \lambda_j\} \right) .
\]

Let \(U \subset \R^n\) be the open convex cone made of all residue trace vectors \(a=(a_1, \ldots, a_n)\) such that \(a_i>0\) for all \(i\) and \(a_j < \sum_{i \neq j} a_i\) for all \(j\). By Proposition \ref{prop:Dunkl} for every pair \((\lambda, a)\) in \(\mn \times U\) there is a unique Dunkl connection \(\nd_{\lambda, a}\) with simple poles at the configuration of lines \(L_i\) represented by \(\lambda\) and residue traces \(a_i\). With this notation, our theorem reads as follows.

\begin{manualtheorem}{1''}\label{thm:main3}
	The pairs \((\lambda, a)\) for which the Dunkl connection \(\nd_{\lambda, a}\) preserves a non-zero Hermitian form make a proper analytic subset of \(\mn \times U\).
\end{manualtheorem}

\begin{proof}
	Write \(Z \subset \mn \times U\) for the set of all \((\lambda, a)\) such that \(\nd_{\lambda, a}\) preserves a non-zero Hermitian form. We show that: \((i)\) \(Z\) is an analytic subset and \((ii)\) \(Z\) is not all of \(\mn \times U\).
	
	\((i)\) The proof of Lemma \ref{lem:locan} repeats verbatim to show that we can locally write \(Z=\{f=0\}\) where \(f = \det Q(\lambda, a)\) where \(Q(\lambda, a)\) is a matrix as in Equation \eqref{eq:Q} whose entries depend analytically on \((\lambda, a)\).
	
	\((ii)\) By Theorem \ref{thm:main} we can
	take \(\lambda_1 \in \C \setminus \{0, 1\}\) and \(a > 0\) such that the Dunkl connection \(\nabla_0\) with simple poles at the lines \(L_1=\{z=0\}\), \(L_2=\{z=w\}\),  \(L_3=\{w=0\}\), \(L_4 =\{z=\lambda_1 w\}\) and equal residue traces \(a_i=a\) for \(1 \leq i \leq 4\) does not preserve any non-zero Hermitian form. Now if \(n \geq 5\) let's fix some arbitrary extra lines \(L_{3+i} = \{z=\lambda_i w\}\) for \(2 \leq i \leq n -3\). Consider the path in \(\mn \times U\) given by \((\lambda, a(t))\) where \(\lambda=(\lambda_1, \ldots, \lambda_{n-3})\) is fixed as above and 
	\[
	a(t) = (a, a, a, a, t, \ldots, t) \hspace{2mm} \text{ for } \hspace{2mm} 0 < t < a .
	\]
	Let \(\nabla_t\) be the continuous (indeed analytic) path of Dunkl connections \(\nd_{\lambda, a(t)}\) for \(t \in (0, a)\). It is clear (from our description of the Dunkl inner product as the hyperbolic barycentre of a weighted configuration) that
	\(\lim_{t \to 0} \nabla_t = \nabla_0\).  This gives us a continuous family of holonomy representations \(\rho(t)\) for \(0 \leq t < a\) and \(\nabla_t\) preserves a non-zero Hermitian form if and only if \(\det Q(t) = 0\) where \(Q(t)\) is a continuous path of positive semi-definite matrices given by Equation \eqref{eq:Qt}. Since \(\nabla_0\) does not preserve any non-zero Hermitian form we have \(\det Q(0) > 0\) and by continuity \(\det Q(t) > 0\) for all \(t > 0\) sufficiently small. This means that \(\nabla_t\) does not preserve any non-zero Hermitian form for small \(t>0\) hence \((\lambda, a(t))\) does not belong to \(Z\).
\end{proof}

\begin{remark}
	The statement of Theorem \ref{thm:main3} remains true in the case when all residue traces are negative \(a_i < 0\) and \(|a_j| < \sum_{i \neq j} |a_i|\) for all \(j\). To prove this one can just repeat the proof above or consider the one parameter family of connections given by scaling all residues  as in Equation \eqref{eq:scalepath}.
\end{remark}

\appendix
\section{Appendix: recollection on spherical metrics.}\label{sec:sphmet}
In this appendix we state (without proof) the results on spherical surfaces needed in this paper.
\begin{definition}\label{def:sphmet}
	Let \(S\) be a surface, let \(x_i\) be points in \(S\) and let \(\alpha_i>0\) be  positive real numbers. A spherical metric on \(S\) with cone angles \(2\pi\alpha_i\) at \(p_i\)  is a smooth constant curvature \(1\) metric on \(S\setminus\cup_i\{p_i\}\) that is isometric to the model \(dr^2 + \alpha_i^2 \sin^2 r d\varphi^2\) in local polar coordinates \((r, \varphi)\) centred at \(p_i\). 
\end{definition}

\begin{remark}
If the surface \(S\) is equipped with an orientation then the spherical metric  induces the structure of a Riemann surface on it with marked points at its conical singularities. 

A (conformal) spherical metric on a Riemann surface \(X\)  with cone angles \(2\pi\alpha_i\) at \(x_i\) is a metric as in Definition \ref{def:sphmet} on the underlying topological surface that on its regular part induces the prescribed conformal structure \(X \setminus \cup_i \{x_i\}\).
\end{remark}

\subsection{Moduli spaces.}\label{ap:ms}
Let \(S\) be some fixed compact connected oriented topological surface of genus \(g \geq 0\). We recall the definition of the moduli space of spherical metrics with prescribed cone singularities as in \cite[Section 6]{EPM}.

\begin{itemize}
	\item \(\MS_{g,n}(\alpha_1, \ldots, \alpha_n)\) is the set of equivalence classes of spherical metrics on \(S\) with cone angles \(2\pi\alpha_1, \ldots, 2\pi\alpha_n\) at a collection of marked points \(p_1, \ldots, p_n\) where two metrics are equivalent if there is an orientation preserving isometry that respects the markings. We endow \(\MS_{g,n}(\alpha_1, \ldots, \alpha_n)\) with the topology induced by the Lipschitz distance between metrics.
	
	\item The Lipschitz distance between two metrics \(g_1, g_2 \in \MS_{g,n}(\alpha_1, \ldots, \alpha_n)\) is
	\begin{equation}
	d(g_1, g_2) = \inf_f \log (\max\{\dil(f), \dil(f^{-1})\})
	\end{equation} 
	where the infimum runs over all orientation preserving bi-Lipschitz homeomorphisms that respect the markings and \(\dil(f)\) denotes the dilation of the map \(f\).\footnote{\(\dil(f)\) is the smallest \(K>0\) such that \(d_{g_2}(f(p), f(q)) \leq K d_{g_1}(p,q)\) for all \(p,q\).}
\end{itemize}

\begin{definition}[\cite{EPM}]
	A \(2\)-marking on a torus \(T\) is the choice of an isomorphism \(H_1(T, \Z_{2}) \cong (\Z_{2})^2\). The space \(\MSm\) is the set of isomorphism classes of \(2\)-marked spherical tori with \(1\) cone point of angle \(2\pi\alpha\) endowed with the Lipschitz topology, where isomorphisms are given by orientation preserving isometries compatible with the \(2\)-markings.
\end{definition}

\begin{theorem}[{\cite[Theorem 4.8]{EPM}}]\label{EMP}
	Let \(\alpha\) be a real number \(>1\), let \(s=(\alpha+1)/2\) and write
	\(m= \lfloor s \rfloor\). If \(s \notin \Z\) then \(\MSm\) is homeomorphic to a compact connected orientable surface \(\Sigma_g\) of genus
	\begin{equation}
	g = \frac{(m-1)(m-2)}{2} 
	\end{equation}
	with \(3m\) punctures.
\end{theorem}

A metric \(\hat{g}\) in \(\MSm\) endows the underlying topological torus with the conformal structure of an elliptic curve \(C\). 
The curve \(C\) is a double branched cover of \(\CP^1\).\footnote{It is proved in \cite[Proposition 2.17]{EPM} that when \(\alpha \notin 2\Z+1\) then the metric \(\hat{g}\) is invariant under the conformal involution of \(C\) given by the deck transformation of the cover \(C \xrightarrow{2:1} \CP^1\).}
The critical point set \(\tilde{S} \subset C\) is made of \(4\) points and the cone point (that we shall denote by \(\tilde{y}_4\)) belongs to \(\tilde{S}\).
The \(2\)-marking gives us a labelling of the three points in \(\tilde{S} \setminus \tilde{y}_4\). Let \(y_i \in \CP^1\) be the image of \(\tilde{y}_i\) under the double cover. Thus we obtain an element \(F(\hat{g}) \in \m\) represented by the ordered \(4\)-tuple of points \((y_1, y_2, y_3, y_4)\). This defines the forgetful map
\begin{equation}\label{eq:forgmap}
F: \MSm \to \m .
\end{equation}
In our case at hand the non-bubbling condition NB \(>0\) of \cite[Definition 1.5]{MPII} reduces to \(\alpha\) being not an odd integer, as a consequence we have the following.
\begin{theorem}[\cite{MPII, EPM}]\label{MP}
	Suppose that \(\alpha \notin 2\Z+1\) then the forgetful map  \eqref{eq:forgmap} is proper and surjective.
\end{theorem}

\subsection{Standard unitary connections and spherical metrics.}\label{sec:unitaryconnect}
We recall the natural correspondence between standard unitary connections on \(\C^2\) and spherical metrics on \(\CP^1\) from  \cite[Section 3]{Pan} (see in particular Theorem 1.8 and Proposition 4.7).

Let \(\nabla\) be a standard unitary connection on \(\C^2\) with poles at  lines \(L_i\) 
\[
\nabla = d - \sum A_i \frac{d\ell_i}{\ell_i}
\]
and let \(h\) be a positive definite Hermitian form preserved by \(\nabla\). Moreover, we will assume that the residue traces \(a_i=\tr A_i \in \R^*\) satisfy the following:
\begin{equation}\label{eq:positivity}
a_i < 1  \hspace{2mm} \text{ for all } i \hspace{2mm} \text{ and } \hspace{2mm} c = \frac{1}{2} \sum_i a_i < 1 . 	
\end{equation}

Let \(x \in \CP^1 \setminus \cup_i \{x_i\}\) where \(x_i\) are the points corresponding to the lines \(L_i\).
Let \(\pi: \C^2 \setminus \{0\} \to \CP^1\)  be the natural projection map and write \(L= \pi^{-1}(x)\). 
Given two tangent vectors \(v, w\) in \(T_x \CP^1\) we can lift them by \(\pi\) to vector fields \(V, W\) along \(L\) which are orthogonal to \(TL\) with respect to \(h\). 
Let \(E = (1-c)^{-1} (z \p_z + w \p_w)\) be the Euler vector field of \(\nabla\).
By homogeneity, the ratio \(h(E, E)^{-1}h(V, W)\) is constant along \(L\). The expression
\begin{equation}\label{eq:sphmet}
	g(v, w) = 4 \cdot \frac{h(V, W)}{h(E, E)}
\end{equation}
defines a metric on \(\CP^1 \setminus \cup_i \{x_i\}\) locally isometric to the \(2\)-sphere of radius \(1\).

\begin{lemma}\label{lem:corr1}
	The metric \(g\) extends over \(x_i\) with cone angle \(2\pi(1-a_i)>0\).
\end{lemma}

\begin{proof}[Sketch proof]
	This local statement follows from the fact that the Hermitian metric \(h\) is isometric  close to \(L_i \setminus \{0\}\) to the product of a \(2\)-cone of total angle \(2\pi(1-a_i)\) with a flat factor \(\R^2\) tangent to \(L_i\). 
\end{proof}

Conversely, we have the following.

\begin{lemma}\label{lem:corr2}
	Let \(g\) be a spherical metric on \(\CP^1\) with cone angles \(2\pi\alpha_i>0\) at points \(x_i\). Then there is a standard unitary connection \(\nabla\) on \(\C^2\) with residue traces \(a_i=1-\alpha_i\) at the lines \(L_i=\pi^{-1}(x_i)\) and a positive definite Hermitian form \(h\) preserved \(\nabla\) such that \(g\) is given by Equation \eqref{eq:sphmet} on its regular part.
\end{lemma}

\begin{proof}[Sketch proof]
	The metric \(g\) with cone angles \(2\pi\alpha_i>0\) at \(x_i \in \CP^1\) lifts through the Hopf map to a constant curvature \(1\) metric \(\bar{g}\) on the \(3\)-sphere with cone angles \(2\pi\alpha_i\) in transverse directions to the Hopf circles lying over the points \(x_i\). The Riemannian cone \(d\rho^2 + \rho^2 \bar{g}\) defines a polyhedral K\"ahler cone metric on \(\C^2\) whose Levi-Civita connection is standard in complex coordinates that linearise the holomorphic Euler vector field \((\rho \p_{\rho})^{1,0}\). 
\end{proof}

\begin{remark}
	Note that the condition \(c < 1\) is equivalent to the Gauss-Bonnet constraint
	\(\chi(S^2) + \sum_i (\alpha_i-1)>0\). 
	This number \(c\) can be interpreted in terms of the intrinsic geometry of the polyhedral K\"ahler cone. Namely,
	the restriction of the polyhedral K\"ahler cone metric on \(\C^2\) to any complex line going through the origin is a \(2\)-cone of total angle \(2\pi(1-c)\).
\end{remark}

\begin{remark}
In the case that \(0<a_i<1\) for all \(i\) so that the cone angles \(2\pi\alpha_i\) are in the interval \((0, 2\pi)\), the 
Troyanov/Luo-Tian Theorem \cite{Troy, LT} asserts that
a spherical metric with on \(\CP^1\) with cone angles \(2\pi\alpha_i\) at \(x_i\) exists if and only if
\[
(1-\alpha_j) < \sum_{i \neq j} (1-\alpha_i) \hspace{2mm} \text{ for all } j.
\]
It is interesting to note that this agrees with the condition for the existence of a Dunkl connection given by Proposition \ref{prop:Dunkl}. 

Our main results shows that, in general, the standard unitary connections corresponding to spherical metrics are \emph{not} Dunkl.
More precisely, for any \(a \in (0, 1/2)\) and \(\lambda \in \m\) there is a unique spherical metric on \(\CP^1\) with cone angle \(2\pi(1-a)\) at \(0,1, \lambda, \infty\). 
Write \(\ns_{\lambda, a}\) for the corresponding standard unitary connections.
With this notation,
Theorem \ref{thm:main} implies that \(\nd_{\lambda, a} \neq \ns_{\lambda, a}\) whenever the pair \((\lambda, a)\) belongs to an open dense subset of \(\m \times (0, 1/2)\).	
\end{remark}

\begin{remark}	
The natural correspondence between standard unitary connections and spherical metrics given by Lemmas	\ref{lem:corr1} and \ref{lem:corr2} is continuous with respect to the natural topology on the space of standard connections\footnote{The space of standard connections with prescribed residue traces is an affine space of dimension \(n-3\) and so it has a natural topology.} and the Lipschitz topology on the space of spherical metrics. Indeed, 
if the monodromy of the spherical metrics is non-coaxial (i.e. irreducible holonomy of the associated standard connections) then the space of spherical surfaces with Lipschitz topology embeds continuously into the \((n-3)\)-dimensional affine space of projective structures (see \cite{MPinprep}). On the other hand, a standard connection induces a projective structure and this defines a continuous map between these the two affine spaces. In the co-axial or reducible holonomy case the correspondence should be taken between spherical metrics and pairs \((\nabla, h)\) where \(h\) is a positive Hermitian form preserved by \(\nabla\).
\end{remark}

\bibliographystyle{alpha}
\bibliography{BIBLIO}

\begin{thebibliography}{GKM00}

\bibitem[Beu07]{Beu}
Frits Beukers.
\newblock Gauss' hypergeometric function.
\newblock In {\em Arithmetic and geometry around hypergeometric functions},
  volume 260 of {\em Progr. Math.}, pages 23--42. Birkh\"{a}user, Basel, 2007.

\bibitem[BH99]{BH}
Martin~R. Bridson and Andr\'{e} Haefliger.
\newblock {\em Metric spaces of non-positive curvature}, volume 319 of {\em
  Grundlehren der Mathematischen Wissenschaften}.
\newblock Springer-Verlag, Berlin, 1999.

\bibitem[CHL05]{CHL}
Wim Couwenberg, Gert Heckman, and Eduard Looijenga.
\newblock Geometric structures on the complement of a projective arrangement.
\newblock {\em Publ. Math. Inst. Hautes \'{E}tudes Sci.}, (101):69--161, 2005.

\bibitem[dBP21]{dBP}
Martin de~Borbon and Dmitri Panov.
\newblock Polyhedral {K}\"{a}hler cone metrics on $\mathbb{C}^n$ singular at
  hyperplane arrangements.
\newblock {\em ar{X}iv:2106.13224}, 2021.

\bibitem[Don11]{don}
Simon Donaldson.
\newblock {\em Riemann surfaces}, volume~22 of {\em Oxford Graduate Texts in
  Mathematics}.
\newblock Oxford University Press, Oxford, 2011.

\bibitem[EMP20]{EPM}
Alexandre Eremenko, Gabriele Mondello, and Dmitri Panov.
\newblock Moduli of spherical tori with one conical point.
\newblock {\em ar{X}iv: 2008.02772, to appear in Geometry and Topology}, 2020.

\bibitem[GKM00]{KapMond}
Daniel Gallo, Michael Kapovich, and Albert Marden.
\newblock The monodromy groups of {S}chwarzian equations on closed {R}iemann
  surfaces.
\newblock {\em Ann. of Math. (2)}, 151(2):625--704, 2000.

\bibitem[IY08]{Yak}
Yulij Ilyashenko and Sergei Yakovenko.
\newblock {\em Lectures on analytic differential equations}, volume~86 of {\em
  Graduate Studies in Mathematics}.
\newblock American Mathematical Society, Providence, RI, 2008.

\bibitem[KLM09]{KLM}
Michael Kapovich, Bernhard Leeb, and John Millson.
\newblock Convex functions on symmetric spaces, side lengths of polygons and
  the stability inequalities for weighted configurations at infinity.
\newblock {\em J. Differential Geom.}, 81(2):297--354, 2009.

\bibitem[LT92]{LT}
Feng Luo and Gang Tian.
\newblock Liouville equation and spherical convex polytopes.
\newblock {\em Proc. Amer. Math. Soc.}, 116(4):1119--1129, 1992.

\bibitem[MP]{MPinprep}
Gabriele Mondello and Dmitri Panov.
\newblock On the moduli space of spherical surfaces with conical points.
\newblock {\em In preparation}.

\bibitem[MP19]{MPII}
Gabriele Mondello and Dmitri Panov.
\newblock Spherical surfaces with conical points: systole inequality and moduli
  spaces with many connected components.
\newblock {\em Geom. Funct. Anal.}, 29(4):1110--1193, 2019.

\bibitem[Pan09]{Pan}
Dmitri Panov.
\newblock Polyhedral {K}\"{a}hler manifolds.
\newblock {\em Geom. Topol.}, 13(4):2205--2252, 2009.

\bibitem[Tho06]{RT}
R.~P. Thomas.
\newblock Notes on {GIT} and symplectic reduction for bundles and varieties.
\newblock In {\em Surveys in differential geometry. {V}ol. {X}}, pages
  221--273. Int. Press, Somerville, MA, 2006.

\bibitem[Thu97]{ThuBook}
William~P. Thurston.
\newblock {\em Three-dimensional geometry and topology. {V}ol. 1}, volume~35 of
  {\em Princeton Mathematical Series}.
\newblock Princeton University Press, Princeton, NJ, 1997.

\bibitem[Tro91]{Troy}
Marc Troyanov.
\newblock Prescribing curvature on compact surfaces with conical singularities.
\newblock {\em Trans. Amer. Math. Soc.}, 324(2):793--821, 1991.

\bibitem[Yos87]{yoshida}
Masaaki Yoshida.
\newblock {\em Fuchsian differential equations}.
\newblock Aspects of Mathematics, E11. Friedr. Vieweg \& Sohn, Braunschweig,
  1987.

\end{thebibliography}

\Address

\end{document}